\documentclass[format=sigconf,vvarbb,pbalance=false,dvipsnames,fleqn]{acmart}
\pdfoutput=1

\usepackage{hyperref}
\usepackage{mathtools}

\usepackage[kw]{pseudo}
\pseudoset{
  compact,
  label=\scriptsize\sffamily\color{gray}\arabic*,
  ctfont=\sffamily\color{gray},
  ct-left=\ctfont{\# },
  ct-right=,
  line-height=1.2,
  indent-length=3ex,
}

\floatstyle{ruled}
\newfloat{algorithm}{ht}{}
\floatname{algorithm}{Algorithm}
\usepackage[shortlabels, inline]{enumitem}

\def\parencite{\citep}
\def\textcite{\citet}
\def\Textcite{\Citet}

\newtheorem{theorem}{Theorem}[section]
\newtheorem{proposition}[theorem]{Proposition}
\newtheorem{lemma}[theorem]{Lemma}

\theoremstyle{remark}
\newtheorem{remark}[theorem]{Remark}

\usepackage{tikz}
\usepackage{pgfplots}
\usetikzlibrary{arrows.meta}
\usetikzlibrary{backgrounds}
\usepgfplotslibrary{patchplots}
\usepgfplotslibrary{fillbetween}
\pgfplotsset{%
    layers/standard/.define layer set={%
        background,axis background,axis grid,axis ticks,axis lines,axis tick labels,pre main,main,axis descriptions,axis foreground%
    }{
        grid style={/pgfplots/on layer=axis grid},%
        tick style={/pgfplots/on layer=axis ticks},%
        axis line style={/pgfplots/on layer=axis lines},%
        label style={/pgfplots/on layer=axis descriptions},%
        legend style={/pgfplots/on layer=axis descriptions},%
        title style={/pgfplots/on layer=axis descriptions},%
        colorbar style={/pgfplots/on layer=axis descriptions},%
        ticklabel style={/pgfplots/on layer=axis tick labels},%
        axis background@ style={/pgfplots/on layer=axis background},%
        3d box foreground style={/pgfplots/on layer=axis foreground},%
    },
}

\setcopyright{none}

\renewcommand{\epsilon}{\varepsilon}
\newcommand{\ud}{\mathrm{d}}
\renewcommand{\st}{\ \middle|\ }
\newcommand{\bx}{\square}

\newcommand{\width}{\operatorname{width}}
\newcommand{\dist}{\operatorname{dist}}
\newcommand{\md}{\operatorname{mid}}
\newcommand{\wprec}{\smash{u_\mathrm{prec}}}
\newcommand{\opnorm}[1]{{\left\vert\kern-0.25ex\left\vert\kern-0.25ex\left\vert \smash{#1}
    \right\vert\kern-0.25ex\right\vert\kern-0.25ex\right\vert}}

\setcopyright{rightsretained}
\acmConference[ISSAC '24]{International Symposium on Symbolic and Algebraic
Computation}{July 16--19, 2024}{Raleigh, NC, USA}
\acmBooktitle{International Symposium on Symbolic and Algebraic Computation
(ISSAC '24), July 16--19, 2024, Raleigh, NC, USA}
\acmDOI{10.1145/3666000.3669673}
\acmISBN{979-8-4007-0696-7/24/07}

\title{Validated Numerics for Algebraic Path Tracking}

\author{Alexandre Guillemot}
\orcid{}
\affiliation[obeypunctuation=true]{
  \institution{Inria, Université Paris-Saclay},
  \city{Palaiseau},
  \country{France}
}

\author{Pierre Lairez}
\orcid{}
\affiliation[obeypunctuation=true]{
  \institution{Inria, Université Paris-Saclay},
  \city{Palaiseau},
  \country{France}
}

\renewcommand{\shortauthors}{A. Guillemot \& P. Lairez}


\begin{abstract}
  Using validated numerical methods, interval arithmetic and Taylor models, we propose a certified predictor-corrector loop for tracking zeros of polynomial systems with a parameter. We provide a Rust implementation which shows tremendous improvement over existing software for certified path tracking.
\end{abstract}

\begin{document}
\maketitle

\section{Introduction}

Path tracking, or homotopy continuation,
is the backbone of numerical algebraic geometry \parencite{SommeseVerscheldeWampler_2005}.
Given  a polynomial map~$(t, x) \mapsto F_t(x)$, from~$\mathbb{C}\times \mathbb{C}^n$ to~$\mathbb{C}^n$,
and a regular zero~$\zeta \in \mathbb{C}^n$ of~$F_0$,
we want to compute for~$t\in [0,1]$ the continuation~$\zeta_t$,
which is, assuming well-posedness, the unique continuous function of~$t$ such that~$F_t(\zeta_t) = 0$ and~$\zeta_0 = \zeta$.
Heuristic approaches to path tracking are sometimes enough
for solving polynomial systems, since we may have the possitiblity to certify zeros of the target system~$F_1$ \emph{a posteriori}.
However, for computing monodromy actions --\,which applies to irreducible decomposition \parencite{SommeseVerscheldeWampler_2001}
or Galois group computation \parencite{HauensteinRodriguezSottile_2017}\,--
and finer invariants --\,such as braids \parencite{RodriguezWang_2017}\,--
it is not sound to use heuristic continuation methods.

\subsubsection*{Contributions}
Building upon Moore's interval arithmetic criterion to isolate a zero of a polynomial system
\parencite{Krawczyk_1969,Moore_1977,Rump_1983},
and following ideas from
\textcite{VanderHoeven_2015}, also developed independently by~\textcite{DuffLee_2024},
we propose a new path tracking algorithm.
The originality of this result lies in the model of computation which rather than idealizing interval arithmetic accounts for what a real software library like MPFI provides.
There is no discrepancy between what is presented and what is implemented.
We prove correction and termination in this model.
In particular, we provide an algorithm (Section~\ref{sec:refin-moore-boxes}) to refine isolating boxes provided by Moore's criterion.
We use this algorithm to formulate the path tracking algorithm (Section~\ref{sec:path-tracking}).
The main difficulty is the balance of the working precision, which should be as low as possible for performance, but large enough to ensure appropriate convergence properties and termination.

The use of Taylor models enables predictors (Section~\ref{sec:predictors}), such as the tangent predictor, or the cubic Hermite predictor,
as observed by \textcite{VanderHoeven_2015},
leading to tremendous improvement over previous methods based on Smale's $\alpha$-theory.
We provide a Rust implementation that we compare with existing software for path tracking, both certified and noncertified (Section~\ref{sec:experiments}).
We find that the number of iterations performed using the Hermite predictor is in the same order of magnitude as noncertified approaches.

\subsubsection*{Related work}
The method of path tracking received vast attention
in the context of numerical multivariate polynomial system solving.
It is the method of choice for most state-of-the-art software:
PHCpack \parencite{Verschelde_1999}, Bertini \parencite{BatesHauensteinSommeseWampler_2013}
or HomotopyContinuation.jl \parencite{BreidingTimme_2018}, for example.
Despite their effectiveness (or perhaps \emph{because} of their effectiveness),
and despite recent work which brings robustness to unprecedented levels \parencite{TelenVanBarelVerschelde_2020},
this software does not guarantee correctness, that is the consistency of the result with the definition of the continuation.
We obtain zeros of the target system which we may certify independently, but we are not certain that the paths were tracked correctly, without swapping for example.

There are countless certified path tracking algorithms
based on Shub and Smale's $\alpha$-theory. For the most part, they were developed for complexity analyses
(for a review, see \textcite{Cucker_2021}) and their implementation is difficult.
\Textcite{BeltranLeykin_2012,BeltranLeykin_2013} took on the challenge within the Macaulay2 package for numerical algebraic geometry \parencite{Leykin_2011}.
In a specific case of “Newton homotopies” (where the system has the form~$F_t(x) = F_1(x) - (1-t) F_1(v)$ for some constant~$v\in \mathbb{C}^n$),
\textcite{HauensteinLiddell_2016,HauensteinHaywoodLiddell_2014} managed to incorporate a tangent predictor into the~$\alpha$-theory
and obtained significant improvement.
In the univariate case ($n=1$),
methods are much more diverse,
even though path tracking is not the method of choice for solving in this case.
We note in particular working implementations by \textcite{Kranich_2015,MarcoBuzunarizRodriguez_2016,XuBurrYap_2018}.

\section{Moore's criterion}

Let~$V$ be a finite dimensional linear space over~$\mathbb{R}$ with a norm~$\|-\|$.
We denote by $\opnorm{-}$ the associated operator norm on~$\mathrm{End}(V)$.
Let~$B$ denote the closed unit ball.
(We will later choose~$V = \mathbb{C}^n \simeq \mathbb{R}^{2n}$ and~$\|-\|$
will be the real $\infty$-norm, so~$B$ will be a box.)

\begin{theorem}
  \label{thm:moore}
  Let~$f : V \to V$ be a continuously differentiable map and let~$\rho \in (0,1)$.
  Let~$x\in V$, $r > 0$, and let~$A: V\to V$ be a linear map.
  Assume that for any~$u, v\in rB$,
  \[  - A  f(x) + \big [ \operatorname{id}_V - A \circ \ud f(x + u) \big ] (v) \in  \rho r B. \]
  Then there is a unique~$\zeta \in x + rB$ such that~$f(\zeta) = 0$.
  Moreover,
  \begin{enumerate}[(i)]
    \item \label{item:rump:zetaloc} $\| x - \zeta \| \leq \rho r$;
  \end{enumerate}
  and for any~$y \in x + rB$,
  \begin{enumerate}[(i), resume]
    \item \label{item:rump:invertibility} $A$  and~$\ud f(y)$ are invertible;
    \item \label{item:rump:triple} $\opnorm{\ud f(y)^{-1} A^{-1}} \leq \left( 1-\rho \right)^{-1}$;
    \item \label{item:rump:triple2} $\opnorm{A} \leq (1+\rho) \opnorm{\ud f(y)^{-1}}$;
    \item \label{item:rump:dist} $(1-\rho) \|y-\zeta\| \leq \|A f(y)\| \leq (1+\rho) \|y - \zeta\|$.
  \end{enumerate}
\end{theorem}

The operator mapping a compact convex set~$E$ containing~$0$ in its interior to~$K(E) = -A f(x) + \left[ \operatorname{id}_V - A\circ \ud f(x+E) \right](E)$,
has been introduced by \textcite{Krawczyk_1969}
to refine isolating boxes.
\Textcite{Moore_1977} showed that the inclusion~$K(E) \subseteq E$
implies the existence of a zero of~$f$ in~$x+E$, by a reduction to Brouwer's fixed-point theorem.
Then \textcite[Theorem~7.4]{Rump_1983} proved unicity of the root if~$K(E) \subseteq \smash{\mathring E}$ (which implies~$K(E) \subseteq \rho E$ for some~$\rho \in (0,1)$ by compactness of~$E$).\footnote{In fact Rump merely requires the strict inclusion~$K(E) \subset E$ but this is not enough to obtain unicity, as shown by a simple linear projection~$f(x, y) = (x, 0)$. \Textcite[Theorem~11]{AlefeldMayer_2000} provide a correct version with a slightly weaker premise than~$K(E) \subseteq \smash{\mathring E}$ but stronger than~$K(E) \subset E$.}
In the variant above,
we assume that~$K(E) \subseteq \rho E$ for some~$\rho \in (0, 1)$, and in addition
we assume that~$E$ is a centered ball with respect to some norm (which is always the case if~$E = -E$, in addition to~$E$ being compact, convex and a neighborhood of~$0$).
In exchange for this extra assumption
we can prove the theorem with Banach's fixed-point theorem, which is more elementary than Brouwer's,
work with the operator norm induced by the norm instead of the spectral radius,
and obtain the quantitative statements~(i)--(v) which will prove useful later.

\begin{proof}[Proof]
  Consider the function~$g(y) = y - A \circ f(y)$.
  The assumption on~$f$ rewrites as
  \begin{equation}
    \label{eq:5}
    \|\pm A  f(x) + \ud g(y) (v)\| \leq \rho r,
    \quad\forall y\in x + rB, \forall v \in rB,
  \end{equation}
  where the~$\pm$ sign comes from changing~$v$ into~$-v$, using~$-B = B$.
  The triangle inequality then implies
  \begin{equation*}
    \label{eq:9}
    2 \| \ud g(y)(v) \| \leq \|\ud g(y) (v)  - A  f(x)\| + \|\ud g(y) (v) + A  f(x)\|,
  \end{equation*}
  which shows that~$\opnorm{\ud g(y)} \leq \rho$, so~$g$ is~$\rho$-Lipschitz continuous.
  Since~$\rho < 1$, $g$ is a contracting map.
  Moreover, it is well known \parencite[Thm. 18.2.1]{Lang_1997}
  that $\opnorm{\ud g(y)} \leq \rho < 1$ implies the invertibility of the operator~$\mathrm{id} - \ud g(y)$, that is~$ A\circ \ud f(y)$,
  which implies~\ref{item:rump:invertibility}.
  The bounds \ref{item:rump:triple} and~\ref{item:rump:triple2} are also easy consequences.

  Let~$u \in E$ and let~$u_t = x + t u$, for~$t\in [0,1]$.
  By integrating the derivative, we compute
   \begin{equation*}\textstyle
     g(x+u) 
            = x + \int_0^1 \big(-Af(x) + \ud g(u_t) (u) \big) \ud t,
   \end{equation*}
   which shows, using~\eqref{eq:5}, that~$g(x+ rB) \subseteq x + \rho r B \subset x + rB$.
   By Banach's theorem, $g$ has a unique fixed point~$\zeta$ in~$x+ rB$.
   Since~$A$ is invertible, $\zeta$ is a zero of~$f$.
   Inequality~\ref{item:rump:zetaloc} follows from~$\zeta$ belonging to~$g(x+rB)\subseteq x+\rho r B$
   and inequalities~\ref{item:rump:dist} follow from the $\rho$-Lipschitz continuity of~$g$.
\end{proof}

\section{Data structures}

\subsection{Arithmetic circuits}

We represent polynomial functions~$\mathbb{C}^n \to \mathbb{C}^m$ as \emph{arithmetic circuits}, also known as \emph{straight-line programs}.
Briefly, an arithmetic circuit with input space~$\mathbb{C}^n$, is a directed acyclic graph, multiple edges allowed,
with four types of nodes:
\begin{enumerate*}
  \item input nodes, with no incoming egdes and labelled with an integer in~$\left\{ 1,\dotsc,n \right\}$;
  \item constant nodes, with no incoming edges and labelled with an element of~$\mathbb{C}$;
  \item addition nodes $+$, with exactly two incoming edges;
  \item multiplication nodes $\times$, with exactly two incoming edges.
\end{enumerate*}
We associate in the obvious way to each node~$\nu$ of a circuit
a polynomial function~$P_\nu : \mathbb{C}^n \to \mathbb{C}$ \parencite[see][for more details]{Burgisser_2000}.
To a tuple of~$m$ nodes of a circuit, we associate a polynomial function~$\mathbb{C}^n\to \mathbb{C}^m$.

This data structure is useful in that it represents not only a polynomial but also a scheme for evaluating it, over~$\mathbb{C}$ or more general objects, for example interval numbers.
Moreover we may use automatic differentiation (in forward or backward mode)
to transform a circuit to another which also computes the derivative of some nodes with respect to some of the input variables.

\subsection{Intervals}

Checking Moore's criterion
requires more than a point evaluation, but information on the image of a set by polynomial map.
Interval arithmetic provides an effective approach to this issue.
There are many ways to model and implement interval arithmetic.
In short, we choose a set~$\mathbb{F} \subset \mathbb{R}$ of representable numbers,
and we define~$\bx \mathbb{R}$ (read ``box $\mathbb{R}$'' or ``interval $\mathbb{R}$'') to be the set of all nonempty compact intervals of~$\mathbb{R}$ with end points in~$\mathbb{F}$.
Lastly, we assume effective binary operations $\boxplus$ and~$\boxtimes$ on~$\bx \mathbb{R}$ such that
for any~$I, J\in \bx \mathbb{R}$ and any~$x\in I$ and~$y\in J$,
$x+y \in I \boxplus J$ and~$xy \in I\boxtimes J$.
For example, we can choose~$\mathbb{F} = \mathbb{Q}$
and define
\begin{align*}
  [a,b] \boxplus [c, d] &= [a+c, b+d], \text{ and}\\
  [a,b] \boxtimes [c, d] &= [\min(ac, ad, bc, bd), \max(ac, ad, bc, bd)].
\end{align*}
These are the usual formula for interval arithmetic, but they are seldom used in this exact form because of the unbearable swell of the binary size of the interval endpoints they produce in any nontrivial computation.
In practice, $\mathbb{F}$ is often the set of finite IEEE-754 64-bits floating-points numbers
and the formula above are implemented with appropriate rounding of the interval endpoints. (And we usually extend~$\bx \mathbb{R}$ with unbounded intervals to cope with overflows.)
We may also choose~$\mathbb{F} = \left\{ a \smash{2^{b}} \st a, b\in \mathbb{Z} \right\}$, the set of dyadic numbers
and round the endpoints of the interval at a given relative precision, which may change in the course of the computation.
This models, for example, the behavior of multiple precision libraries such as MPFI \parencite{RevolRouillier_1999}
or Arb \parencite{Johansson_2017}.

We define~$\bx \mathbb{C}$ as pairs of elements of~$\bx \mathbb{R}$, representing real and imaginary parts, endowed with the obvious extensions of~$\boxplus$ and~$\boxtimes$.
We consider also vectors of boxes, denoted~$\bx \mathbb{C}^n$.

\subsection{Interval extensions, adaptive precision}

Let~$f$ be a function $\mathbb{C}^n \to \mathbb{C}^m$.
An \emph{interval extension} of~$f$ is a function $\bx f : \bx \mathbb{C}^n \to \bx \mathbb{C}^m$
such that for any~$X \in \bx \mathbb{C}^n$ and any~$x\in X$,
$f(x) \in \bx f(X)$.
If~$f$ is a polynomial function represented by a circuit with constants in~$\mathbb{F}$, then we obtain naturally
an extension of~$f$ by replacing any constant~$c$ in the circuit by the singleton interval~$[c,c]$ and evaluating the circuit using interval operations~$\boxplus$
and~$\boxtimes$.

The correctness of our algorithm does not depend on any hypothesis on interval arithmetic and extensions
other than the basic requirements on interval extensions.
As for termination, we need stronger hypotheses.
As pointed out above, interval arithmetic is usually not implemented in exact arithmetic.
It is also clear that finitely many representable numbers, as provided by the IEEE-754 arithmetic model, will not be enough to express termination arguments based on topology and convergence.
So we introduce an adaptive precision model which can be implemented using any multiple precision interval library.
The operations in this model depend on a parameter $\wprec \in (0,1)$, the \emph{unit roundoff}, which can be raised or lowered at will.
We require that for any~$M \geq 1$ and any~$[a,b], [c, d] \in \bx \mathbb{R}$ included in~$[-M, M]$,
\begin{align}
  \label{eq:6} [a, b] \boxplus [c, d] &\subseteq [a + c - M\wprec, b + d + M\wprec], \text{ and} \\
  [a,b] \boxtimes [c, d] &\subseteq [\min(ac, ad, bc, bd) - M^2 \wprec, \notag\\
  & \hspace{4em} \max(ac, ad, bc, bd) + M^2 \wprec].
  \label{eq:7}
\end{align}
Interval arithmetic implemented with IEEE-754 arithmetic satisfies these contraints, with~$\wprec \sim 2^{-53}$, unless overflow occurs \parencite{MullerBrunieDeDinechinJeannerodJoldesEtAl_2018}.

If we have a circuit~$f: \mathbb{C}^n \to \mathbb{C}^m$,
its interval extension~$\bx f$ depends on the working precision~$\wprec$.
To formulate a useful property, we need some metrics.
Let~$\|-\|$ be the real $\infty$-norm on~$\mathbb{C}^n$ seen as~$\mathbb{R}^{2n}$, that is
\begin{equation}
  \label{eq:20}
  \| (z_1, \dotsc, z_n) \| = \max_{1\leq i \leq n} \max \left( |\operatorname{Re}(z_i)|, |\operatorname{Im}(z_i)| \right).
\end{equation}
For~$X, Y \in \bx \mathbb{C}^n$,
we define the \emph{width} (or \emph{diameter}) and the \emph{magnitude} $\|X\|_\square$ as
\[ \width(X) = \sup_{x, y \in X} \|x - y\|, \ \ \text{and } \|X\|_\square = \sup_{x\in X} \|x\|, \]
and we define the Hausdorff distance by
\[ \dist(X, Y) = \max \big\{
  \mathop{\mathrm{sup}}\limits_{x\in X}
  \mathop{\mathrm{\vphantom{p}inf}}\limits_{y\in Y}
  \|x-y\|,
  \mathop{\mathrm{sup}}\limits_{y\in Y}
  \mathop{\mathrm{\vphantom{p}inf}}\limits_{x\in X}
  \|x-y\|\big\}. \]
Lastly, $\md(X)$ denotes the \emph{midpoint} of~$X$, which is some representable element of~$X$ (and we usually choose one that is as close as possible to the mathematical center).

\begin{proposition}\label{prop:lipschitz-continuity}
  Let~$f : \mathbb{C}^n\to \mathbb{C}^m$ be a circuit with representable constants
  and let~$\bx f$ be its natural interval extension in the adaptive precision model.
  For any compact~$K \subseteq \mathbb{C}^n$,
  there is some $L \geq 0$, independent of the unit roundoff, such that for any~$X, Y \in \bx \mathbb{C}^n$
  included in~$K$:
  \begin{enumerate}[(i)]
    \item\label{item:lipcont:width} $\width ( \bx f(X)) \leq L \left( \width(X) + \wprec \right)$,
    \item\label{item:lipcont:dist} $\dist \left( \bx f(X), \bx f(Y) \right) \leq L \left( \dist(X, Y) + \wprec \right)$.
  \end{enumerate}
\end{proposition}

\begin{proof}
  These two properties are stable under composition, so it is enough to check them for~$\boxplus$ and~$\boxtimes$,
  and this follows directly from~\eqref{eq:6} and~\eqref{eq:7}.
\end{proof}

\subsection{Moore boxes}

Let~$B \subseteq \mathbb{C}^n$ be the closed unit ball for the real~$\infty$ norm~$\|-\|$.
Given a polynomial map~$f : \mathbb{C}^n \to \mathbb{C}^n$ and some~$\rho \in (0,1)$,
a \emph{$\rho$-Moore box for~$f$} is a triple~$(x, r, A) \in \mathbb{C}^n \times \mathbb{R}_> \times \mathbb{C}^{n\times n}$ such that
\begin{equation*}
  \label{eq:11}
  \left\|\, - r^{-1} A f(x) + (I_n - A \cdot \ud f(x + rB)) \cdot B \,\right\| \leq \rho.
\end{equation*}
We say that~$\rho$ is the \emph{contraction factor}.
By Theorem~\ref{thm:moore}, if~$(x, r, A)$ is a Moore box, then there is a unique zero of~$f$ in the set~$x + rB$, called the \emph{associated zero}.
We use Moore boxes (with representable~$x$, $r$ and~$A$) as a data structure to represent a regular zero of~$f$.

Given interval extensions~$\bx f$ and~$\bx \ud f$,
if it holds in interval arithmetic that
\[ \left\| -r^{-1} A \cdot \bx f(x) + \left( I_n - A \cdot \bx \ud f(x + r B) \right) \cdot B  \right\|_\square \leq \rho, \]
then~$(x, A, r)$ is a $\rho$-Moore box for~$f$.
Note that the magnitude is always a representable number, so we can indeed check this inequality accurately.
This leads to Algorithm~\ref{algo:M}.

\begin{algorithm}[t]
  \raggedright
  \caption{Interval certification of a Moore box}
  \begin{description}
    \item[input]
          $\bx f : \bx \mathbb{C}^n \to \bx \mathbb{C}^n$;
          $\bx \ud f : \bx \mathbb{C}^n \to \bx \mathbb{C}^{n\times n}$;\\
          $x \in \mathbb{C}^n$;
          $r \in \mathbb{R}_>$;
          $A \in \mathbb{C}^{n\times n}$;
          $\rho \in (0,1)$

    \item[output] a boolean
  \end{description}
  \begin{pseudo}
    def \fn{M}(\bx f, \bx \ud f, x, r, A, \rho):\\+
      $K \gets -  r^{-1} A \cdot \bx f(x) + \left( I_n - A \cdot \bx\ud f(x + r B) \right)\cdot B$\\
      return $\|K\|_\square \leq \rho$
  \end{pseudo}
  \label{algo:M}
\end{algorithm}

\begin{lemma}[Correctness of Algorithm~\ref{algo:M}]\label{lem:moore-interval}
  Let~$f : \mathbb{C}^n\to \mathbb{C}^n$ be a polynomial function,
  and let~$\bx f : \bx \mathbb{C}^n \to \bx \mathbb{C}^n$ and~$\bx \ud f : \bx \mathbb{C}^n \to \bx\mathbb{C}^{n\times n}$ be interval extensions of~$f$ and~$\ud f$ respectively.
  For any $x\in \mathbb{C}^n$, $A\in \mathbb{C}^{n\times n}$, $r > 0$ and~$\rho \in (0,1)$,
  if~\fn{M}(\bx f, \bx \ud f, x, r, A, \rho) returns \emph{True}, then~$(x, r, A)$ is a~$\rho$-Moore box for~$f$.
\end{lemma}

\section{Refinement of Moore boxes}
\label{sec:refin-moore-boxes}

Following the original idea of \textcite{Krawczyk_1969}, who introduced his operator to refine isolating boxes,
we can refine a Moore box~$(x, r, A)$ for a polynomial map~$f:\mathbb{C}^n\to \mathbb{C}^n$ by computing the intersection
\[ (x + rB) \cap \left( x - A \cdot \bx f (x) + (I_n - A \cdot \bx \ud f(x + rB)) \right). \]
Unfortunately, this cannot always work, because the interval arithmetic may be too gross.
It is possible that~$(x, r, A)$ is a Moore box, but the intersection above is just~$x + rB$, providing no useful refinement.
In other words, it is possible that~$(x, r, A)$ is a $\rho$-Moore box
but \fn{M}(\bx f, \bx \ud f, x, r, A, \rho) returns \emph{False}.

\subsection{Algorithm}

We are given (a circuit representing) a polynomial map~$f : \mathbb{C}^n \to \mathbb{C}^n$,
a $\rho$-Moore box~$(x, r, A)$ with an associated zero~$\zeta$ of~$f$,
and some~$\tau \in (0,1)$.
We want to compute another Moore box with same associated zero and contraction factor~$\tau$.


Algorithm~\ref{algo:refine} proceeds as follows.
The main loop (Line~\ref{line:refine:mainloop})
maintains a triple~$(y, s, U)$
and stops when
the interval arithmetic criterion certifies that~$(y, s, U)$ is
a $\tau$-Moore box.
After the first iteration, the matrix~$U$ is always~$\ud f(y)^{-1}$, give or take some roundoff errors.
The computation of~$U$ need not be performed in interval arithmetic (this is the practical appeal of Krawczyk's operator), the correctness of the algorithm does not depend on the accuracy of this computation, but to ensure termination, the distance from~$U$ to~$\ud f(y)^{-1}$ must go to zero
as the working precision increases (that is~$\wprec \to 0$).
After the main loop, we know that~$(y, s, U)$ is a $\tau$-Moore box.
Before returning it, we check if, by chance, $(y, 2s, U)$ is also a $\tau$-Moore box.
We double the radius until it is not the case (or the radius exceeds~1) and then return the Moore box.

When Moore's criterion cannot be checked,
we try to improve the triple~$(y, s, U)$, either by using a quasi-Newton iteration $y \to y - A \cdot f(y)$, or by shrinking the box with $s\to \frac12 s$.
The choice depends on~$\| A \cdot f(y) \|$, the size of the quasi-Newton step,
compared to~$\tau s$, with the goal of balancing the two terms in Moore's criterion.

We want to run the computation with standard double precision as much as possible,
but the algorithm may warn that the working precision is not large enough, or equivalently, $\wprec$ is not small enough (Lines~\ref{line:refine:uprec1} and~\ref{line:refine:uprec2}).
The computational model assumes that we can increase the working precision, but on the practical side, if only double precision is available (this is the case in the implementation we propose), we abort the computation on a precision warning. In this view, it is important to avoid undue warnings.

In Algorithm~\ref{algo:refine}, two checks may trigger a precision warning.
First, after shrinking the box, when~$\frac{s}{\tau r}$ drops below some threshold.
This is because we expect~$\tau r$ to be the radius of the~$\tau$-Moore box that we are looking for
(simply by considering a degree~2 approximation of~$f$ around~$x$).
So when~$\frac{s}{\tau r}$ is too small we may suspect that Moore's criterion failed because of roundoff errors.
In this case, we want to increase working precision in such a way that~$\wprec$ goes to~0 faster than~$s$.
Second, when performing a quasi-Newton iteration, we check that the roundoff error, that is~$\width(y - \delta)$, is significantly smaller than the size of the quasi-Newton step. This ensures the convergence of~$y$ (Lemma~\ref{lem:convergence-y}).

\begin{remark}
  Algorithm~\ref{algo:refine} features some arbitrary constants
  for which we picked explicit values. Let us name them:
  $\rho = \smash{\frac 78}$, the contraction factor of the input;
  $\alpha = \smash{\frac{1}{64}}$, used to compare~$\|\delta\|_\square$;
  $\lambda = \smash{\frac12}$, used to shrink~$s$;
  $\beta = \smash{\frac{1}{40}}$, used in the precision check.

  Naturally, there is some flexibility in the choice of these constants.
  To keep the algorithm correct, we need to have~$\lambda < 1$ obviously.
  Values closer to~1 will produce bigger boxes with more iterations.
  The smaller~$\alpha$, the closer~$y$ is to the exact root.
  The quasi-Newton iteration converges rapidly, so there is little performance penalty in lowering~$\alpha$,
  but it may cause the precision check to fail earlier.
  For correctness, we need~$\alpha + \rho < 1$.
  (cf. proofs of Lemmas~\ref{lem:track:y} and~\ref{lem:track}).
  The parameter~$\rho$ can be close to~1, it improves performance to allow for Moore boxes with a larger contraction factor. We do not have to worry about the speed of convergence of the quasi-Newton iteration, because interval arithmetic is typically pessimistic. So if we can check a Moore box with contraction factor~$\frac 78$, then it probably has a much lower actual contraction factor. However, other constants degrade if~$\rho$ is too close to~1.
  For correctness, we need
  $\smash{\frac{\rho + \beta}{1-\beta}} < 1$.
  This number is the geometric ratio in Lemma~\ref{lem:convergence-y}.
\end{remark}

\begin{remark}
  Checking Moore's criterion is costly. From the practical point of view, with the path tracking algorithm in mind, it is beneficial for performance to perform first one or two quasi-Newton iterations (with the appropriate precision check) before entering the main loop.
\end{remark}

\begin{algorithm}[t]
  \raggedright
  \begin{description}
    \item[input]
          $f : \mathbb{C}^n \to \mathbb{C}^n$;
          $(x, r, A)$, a $\frac{7}{8}$-Moore box;
          $\tau \in (0,1)$

    \item[output] a $\tau$-Moore box with same associated root as~$(x, r, A)$

  \end{description}
  \begin{pseudo}
    def \fn{refine}(f, x, r, A, \tau):\\+
      $y \gets x$; \quad
      $s \gets r$; \quad
      $U \gets A$; \\
      \label{line:refine:mainloop}while not \fn{M}(\bx f, \bx \ud f, y, s, U, \tau):\\+
        $\delta \gets A \cdot \bx f(y)$\\
        if $\| \delta \|_\square \leq \frac{1}{64} \tau s$: \label{line:refine:deltasmall} \\+
          \label{line:refine:shrink}$s\gets \frac12 s$\\
          if $s < \frac1{16} \tau r$:\\+
            \tn{increase precision enough so that $\wprec = o(s)$} \label{line:refine:uprec1}\\--
        elif $\width(y-\delta) > \tfrac{1}{40} \|\delta\|_\square$: \quad\ct{precision check}\label{line:refine:check}\\+
          \tn{increase working precision} \label{line:refine:uprec2}\\-
        else: \\+
          \label{line:refine:newton}$y \gets \md(y - \delta)$\\-
          $U \gets \md(\bx \ud f(y))^{-1}$ \hfill\ct{unchecked arithmetic} \label{line:refine:U}\\-

      \label{line:refine:growloop}while $2s \leq 1$ and \fn{M}(\bx f, \bx \ud f, x, 2s, U, \tau):\\+
          $s \gets 2s$\\-
      return $y$, $s$, $U$\\-
  \end{pseudo}
  \caption{Refinement of a Moore box}
  \label{algo:refine}
\end{algorithm}

\subsection{Analysis}


\begin{theorem}
  Let~$f : \mathbb{C}^n \to \mathbb{C}^n$ be a circuit,
  let~$(x, r, A)$ be a $\frac{7}{8}$-Moore box for~$f$
  and let~$\tau \in (0, 1)$.
  On input~$x$, $r$, $A$ and~$\tau$, Algorithm~\ref{algo:refine} terminates and
  outputs a $\tau$-Moore box for~$f$ with same associated zero as the input box.
\end{theorem}

The algorithm does \emph{something} until \fn{M}(\bx f, \bx\ud f, y, r, U, \tau) holds.
Then a second loop does \emph{something} that preserves this property.
So it is obvious, by design, that the algorithm outputs a $\tau$-Moore box.
It remains to check that the associated zeros of the input and output are the same,
and that the algorithm terminates.

We first study the sequence~$y_0 = x, y_1, \dotsc$
where~$y_k$ is the value of~$y$ after the $k$th quasi-Newton iteration.
In principle, the quasi-Newton iteration converges, but roundoff errors could take over.
The precision check (Line~\ref{line:refine:check}) ensures that it does not happen.

\begin{lemma}\label{lem:convergence-y}
  For any~$k \geq 0$, $\|y_k - x\|\leq r$ and $\|y_k - \zeta\| \leq \frac{7}{8} \big( \frac{12}{13} \big)^k r$.
\end{lemma}

\begin{proof}
  We prove the claim by induction.
  First, $y_0 = x$, so we have $\|y_0 - x\| = 0$ and $\|y_0 - \zeta\| \leq  \frac{7}{8} r$, by Theorem~\ref{thm:moore}\ref{item:rump:zetaloc}.
  Then, by definition, $y_{k+1} = \md( y_k - \delta_k)$,
  where~$\delta_k  = A \cdot \bx f(y_k)$ computed in interval arithmetic.
  So there is some~$\epsilon_k \in \mathbb{C}^n$ with~$\|\epsilon_k\| \leq \width(y_k - \delta_k)$ such that
  \begin{equation}
    \label{eq:14}
    y_{k+1} = y_k - A f(y_k) + \epsilon_k = g(y_k) + \epsilon_k.
  \end{equation}
  Recall that~$g : y \mapsto y - A f(y)$ is a~$\tfrac{7}{8}$-Lipschitz continuous function on~$x+rB$ such that~$g(\zeta) = \zeta$
  and~$g(x+ rB) \subseteq x + \smash{\tfrac78}B$.
  In particular,
  \begin{gather}
    \label{eq:23}
    \|y_{k+1} - \zeta\| \leq \tfrac78 \|y_k - \zeta\| + \|\epsilon_k\|,\quad \text{and}\\
    \label{eq:24}
    \|y_{k+1} - x\| \leq \tfrac 78 r + \|\epsilon_k\|.
  \end{gather}
  The precision check (Line~\ref{line:refine:check}) implies that
  \begin{equation}
    \label{eq:12}
    \|\epsilon_k\| \leq \width(y_k - \delta_k) \leq \tfrac{1}{40} \|\delta_k\|_\square.
  \end{equation}
  Moreover $\width(\delta_k) \leq \width(y_k - \delta_k)$ so,
  \begin{align*}
    \|\delta_k\|_\square &\leq \| A f(y_k) \| + \width(\delta_k)
                         \leq \| A f(y_k) \| + \tfrac{1}{40} \|\delta_k\|_\square,
  \end{align*}
  and by Theorem~\ref{thm:moore}\ref{item:rump:dist} it follows
  \begin{equation}
    \label{eq:16}
    \|\delta_k\|_\square \leq \tfrac{40}{39} \|A f(y_k)\| \leq \tfrac{25}{13} \|y_k - \zeta\|.
  \end{equation}
  In combination with~\eqref{eq:12}, we obtain
  \begin{equation} \label{eq:13}
    \| \epsilon_k \| \leq \tfrac{5}{104} \| y_k - \zeta \|.
  \end{equation}
  It follows easily from~\eqref{eq:23} and~\eqref{eq:13} that
  \begin{equation}
    \label{eq:15}
    \|y_{k+1}-\zeta\|\leq \tfrac{12}{13}\|y_k - \zeta\|.
  \end{equation}
  From the induction hypothesis, we have~$\|y_k - \zeta\| \leq \smash{\frac78} r$
  and it follows from~\eqref{eq:24} and \eqref{eq:13} that~$\|y_{k+1} - x\| \leq r$, which proves the induction step.
\end{proof}

\begin{lemma}\label{lem:track:y}
  At any point of the algorithm,
  $\|y - \zeta\| \leq s$.
\end{lemma}

\begin{proof}
  The inequality obviously holds at the start of the algorithm.
  It remains to check that it still holds when~$y$ is moved by a quasi-Newton iteration or when~$s$ is halved.
  The norm~$\|y - \zeta\|$ decreases when a quasi-Newton iteration is performed, by~\eqref{eq:15}. So the inequality is preserved in this case.
  In the second case, we note that Line~\ref{line:refine:shrink}
  is only reached when~$\|\delta\|_\square \leq \frac{1}{64} \tau s$.
  Since~$\|A f(y)\| \leq \|\delta\|_\square$,
  this implies, together with Theorem~\ref{thm:moore}\ref{item:rump:dist}, that
  $\|y - \zeta\| \leq 8 \| A f(y) \| \leq \tfrac{1}{8} s$.
  So after~$s \gets \frac s2$, we have $\|y - \zeta\| \leq \frac14 s$, and the inequality holds.
\end{proof}

Since~$\zeta \in y + sB$, it is clear that the Moore box output by the algorithm
is associated to~$\zeta$, by unicity of the associated root.
It only remains to settle termination.

\begin{lemma}\label{lem:refine-technical}
  Considering the values of~$y$, $s$, $U$ and~$\wprec$ in an infinite run of Algorithm~\ref{algo:refine},
  we have:
  \begin{enumerate*}[(i)]
    \item\label{item:4} $s \to 0$;
    \item\label{item:5} $\wprec  = o(s)$;
    \item\label{item:2} $y \to \zeta$;
    \item\label{item:3} $U \to \ud f(\zeta)^{-1}$;
    \item\label{item:1} $\| U f(y) \| \leq \frac 12 \tau s$ eventually.
  \end{enumerate*}
\end{lemma}

\begin{proof}
  First, we show that~$s \to 0$.
  Assume it does not,
  implying that after a certain point, $s$ is never halved,
  so we always have
  \begin{equation}\label{eq:17}
    \|\delta\|_\square > \tfrac{1}{64} \tau s.
  \end{equation}
  Since~$s$ is not halved, every iteration of the main loop fails the ``if'' condition, so it reaches reaches the ``elif'' condition: the precision check.
  It may fail it (and the working precision is raised)
  or pass it (and a quasi-Newton iteration is performed).
  The precision check cannot fail for ever.
  Indeed, every fail causes the working precision to increase
  so~$\width(y - \delta) \to 0$, by Proposition~\ref{prop:lipschitz-continuity}\ref{item:lipcont:width}.
  By~\eqref{eq:17}, $\|\delta\|_\square$ is bounded below,
  so we have $\width(y - \delta) \leq \frac{1}{40}\|\delta\|_\square$
  and the precision check passes.
  This implies that infinitely many quasi-Newton iterations are perfomed.
  By~\eqref{eq:16}, this implies that~$\|\delta\|_\square \to 0$, in contradiction with~\eqref{eq:17}.
  Therefore, $s\to 0$, proving~\ref{item:4}.
  When~$s\to 0$, Line~\ref{line:refine:uprec1}
  ensures that~$\wprec = o(s)$.
  This checks~\ref{item:5}.

  The radius~$s$ is only halved when~$\|\delta\|_\square \leq \frac{1}{64} \tau s$.
  Since~$s$ is repeatedly halved, this implies that~$\|\delta\|_\square \to 0$,
  Therefore $A f(y) \to 0$, and, by Theorem~\ref{thm:moore}\ref{item:rump:dist}, $y \to \zeta$, proving~\ref{item:2}.
  Since $\wprec \to 0$ and~$U = \ud f(y)^{-1}$, up to roundoff errors,
  we also have~$U \to \ud f(\zeta)^{-1}$, proving~\ref{item:3}.
  So it only remains to check~\ref{item:1}.
  We decompose
  \begin{equation}
    \label{eq:8} U f(y) =
    \underbrace{\big( U - \ud f(y)^{-1} \big)}_{= O(\wprec)} \cdot \underbrace{f(y)}_{\to 0} + \underbrace{\big(\ud f(y)^{-1} A^{-1}\big)}_{\opnorm{-} \leq 8} \cdot \underbrace{A f(y)}_{\|-\|\leq \frac{1}{64}\tau s},
  \end{equation}
  using Theorem~\ref{thm:moore}\ref{item:rump:triple}.
  This proves that
  $\| Uf(y)\| \leq \tfrac{1}{8} \tau s + o(s)$.
  So eventually~$\| U f(y) \| \leq \frac 12 \tau s$. This concludes the proof.
\end{proof}

It is now easy to prove that Algorithm~\ref{algo:refine} terminates.
Let $e = s^{-1} U \cdot \bx f(y)$.
Both~$U$ and~$y$ converges (Lemma~\ref{lem:refine-technical}\ref{item:2} and~\ref{item:3}), in particular they are bounded.
By Proposition~\ref{prop:lipschitz-continuity},
$\width(U\cdot \bx f(y)) = O(\wprec)$, and by Lemma~\ref{lem:refine-technical}\ref{item:5}, this is~$o(s)$.
So after division by~$s$, we have~$\width(e) = o(1)$.
Since $s^{-1} U f(y) \in e$,
we have
\begin{equation}
  \label{eq:22}
  \| e \|_\square \leq s^{-1} \| U f(y) \| + \width(e) \leq \tfrac12 \tau + o(1).
\end{equation}
Moreover, $y + s B \to \left\{ \zeta \right\}$ in the Hausdorff metric
because~$y\to \zeta$ and~$s\to 0$
(Lemma~\ref{lem:refine-technical}\ref{item:2} and~\ref{item:4}).
Since $\wprec \to 0$,
Proposition~\ref{prop:lipschitz-continuity} implies that
\[ (I_n - U \cdot \bx \ud f(y + sB)) \cdot B \to \big( I_n - \ud f(\zeta)^{-1} \cdot \ud f(\zeta) \big) \cdot B = 0. \]
It follows that
$\left\| -e + (I_n - U \cdot \bx \ud f(y + sB)) \cdot B  \right\|_\square \leq  \|e\|_\square + o(1).$
By~\eqref{eq:22}, this is eventually less than~$\tau$, which means that Moore's criterion $M(\bx f, \bx \ud f, y, s, U, \tau)$ holds, and the main loop terminates.
Due to the condition~$s \leq \frac12$, the second loop, that tries growing~$s$, also terminates.

\begin{algorithm}[b]
  \raggedright
  \begin{description}
    \item[input]  $F_\bullet$, a circuit~$\mathbb{C}\times \mathbb{C}^{n} \to \mathbb{C}^n$;
          $(x, r, A)$, a $\frac78$-Moore box for~$F_0$
    \item[output] a Moore box for~$F_1$
  \end{description}
  \begin{pseudo}
    def \fn{track}(F_\bullet, x, r, A):\\+
      $t\gets 0$; \quad      $h \gets 1$\\
      while $t < 1$: \label{line:refine:main-loop}\\+
        $x, r, A \gets \fn{refine}(F_t, x, r, A, \frac 18)$ \label{line:track:refine}\\
        $h \gets 2h$; \quad $T \gets [t, t + h]$ \\
        \label{line:refine:inner-loop}while not \fn{M}(\bx F_T, \bx \ud F_T, x, r, A, \frac78):\\+
          $h \gets \frac12 h$; \quad $T \gets [t, t + h]$\\
          if $\wprec > h$:\\+
            \tn{increase working precision}\\--
          $t \gets \sup T$\\-

    return $x, B$
  \end{pseudo}
  \caption{Path tracking}
  \label{algo:track}
\end{algorithm}

\section{Path tracking}
\label{sec:path-tracking}

\subsection{Setting}\label{sec:setting}

We are given an arithmetic circuit~$F : \mathbb{C} \times \mathbb{C}^{n} \to \mathbb{C}^n$.
The first argument is the parameter and put in subscript, so that~$F_t$ denotes the map~$\mathbb{C}^n \to \mathbb{C}^n$ obtained from~$F$ by specialization of the parameter.
It also denotes the circuit obtained by replacing the input nodes \#1 (the index of the parameter variable), with a constant node~$t$.

Let~$\zeta \in \mathbb{C}^n$ be a regular zero of~$F_0$,
that is~$F_0(\zeta) = 0$ and assume that the~$n\times n$ matrix~$\ud F_0(\zeta)$ is invertible.
There is a unique open interval~$I \subseteq \mathbb{R}$ containing~$0$
and a unique continuous function~$Z : I \to \mathbb{C}^n$ such that:
\begin{enumerate}[(i)]
  \item $Z_0 = \zeta$;
  \item $F_t(Z_t) = 0$ for any~$t\in I$;
  \item \label{it:boundary-behavior} if $b \in \overline{I} \setminus I$, then either
  \begin{enumerate}[(a)]
    \item $\lim_{t\to b} \| Z_t\| = \infty$; or
    \item $\lim_{t\to b} \det \left( \ud F_t(Z_t) \right) = 0$.
  \end{enumerate}
\end{enumerate}
This follows from the study of the differential equation
\begin{equation}
  \label{eq:1}
\textstyle  \frac{\ud}{\ud t} Z_t = - \ud F_t(Z_t)^{-1} \cdot \dot F_t( Z_t )
\end{equation}
obtained from the equation~$F_t(Z_t) = 0$ by differentiation,
and where~$\dot F_t$ denote the partial derivative of~$F_t$ with respect to~$t$.
The existence and uniqueness of a local solution is given by the Picard-Lindelöf Theorem \parencite[Theorem~19.1.1]{Lang_1997}.
There is a unique maximal solution interval, this is~$I$, and
the condition~\ref{it:boundary-behavior} reflects the behavior of the solution at the boundary of the maximal interval of definition \parencite[Theorem~19.2.4]{Lang_1997}: the solution diverges or leaves the domain of definition of the differential equation.

In what follows, we assume that~$\zeta$ is given by a $\frac78$-Moore box, and we aim at computing a Moore box for~$Z_1$, as a zero of~$F_1$,
assuming that~$1 \in I$.

\subsection{Algorithm}

Algorithm~\ref{algo:track}
performs the path tracking operation, as defined above,
using the \fn{refine} algorithm.
The main ingredient is the use of Algorithm~\ref{algo:M}
with interval functions~$\mathcal{F}$ and~$\ud\mathcal{F}$
that are exensions of~$F_t$ and~$\ud F_t$ respectively
for a range of values of~$t$.

More precisely, let~$T \in \bx \mathbb{R}$.
In the circuit representing~$F_\bullet$ and~$\ud F_\bullet$,
replace the input nodes~\#1 (the index of the parameter variable) by constant nodes containing the interval value~$T$.
We obtain circuits that can be evaluated over~$\bx \mathbb{C}^n$.
We denote~$\bx F_T$ and~$\bx \ud F_T$ these circuits.
The fundamental property of interval arithmetic
guarantees that the interval functions defined by~$\bx F_T$ and~$\bx \ud F_T$
are interval extensions of~$F_t$ and~$\ud F_t$ respectively, for any~$t\in T$.
In particular, if \fn{M}(\bx F_T, \bx \ud F_T, x, r, A, \frac78) returns \emph{True},
then $(x, r, A)$ is a~$\smash{\frac78}$-Moore box for~$F_t$ for any~$t\in T$.
This follows from Lemma~\ref{lem:moore-interval}, applied with~$f = F_t$
and the interval extensions~$\bx f = \bx F_T$ and~$\bx \ud f = \bx \ud F_T$.

Based on this idea,
given a~$t \in [0,1]$ and a $\smash{\frac78}$-Moore box for~$F_t$,
we refine it into a~$\smash{\frac18}$-Moore box~$(x, r, A)$
then we search for a interval~$T = [t, t+\delta]$
such that~$M(\bx F_T, \bx\ud F_T, x, r, A, \smash{\frac78})$ holds.
If the search is successful, we know that~$(x, r, A)$ is a $\frac78$-Moore box for~$F_{t+\delta}$
and we can repeat the process.

The correctness of the algorithm is glaring, but termination is not.
Can we find a positive $\delta$ at each step? Does the process eventually reach~$t=1$? Or may it converge
to a lower value of~$t$?

\subsection{Analysis}

\begin{theorem}\label{thm:track}
  Let~$F_\bullet : \mathbb{C}\times \mathbb{C}^n\to \mathbb{C}^n$ be a circuit.
  Let~$(x, r, A)$ be a $\frac78$-Moore box for~$F_0$ with associated zero~$\zeta$.
  Let~$I\subseteq \mathbb{R}$ and~$Z : I \to \mathbb{C}^n$ be defined as in Section~\ref{sec:setting}.

  Algorithm~\ref{algo:track} terminates
  if and only if\/~$1\in I$. In this case,
  it outputs a $\frac78$-Moore box for~$F_1$
  with associated zero~$Z_1$.
\end{theorem}

We first prove the termination of the inner ``while'' loop.
Assume it does not terminate.
Let~$K_T \in  \bx \mathbb{C}^n$ denote the box vector
computed in \fn{M}(\bx F_T, \bx \ud F_T, x, r, A, \smash{\frac78})
and let~$K_t$ denote the one that would be computed in \fn{M}(\bx F_t, \bx \ud F_t, x, r, A, \frac18).
Because the loop does not terminate, we always have
$\| K_T \|_\square > \smash{\tfrac 78}$.
However, the triple~$(x, r, A)$ comes from \fn{refine}, with contraction factor~$\frac18$.
This procedure checks~\fn{M}(\bx F_t, \ud \bx F_t, x, r, A, \smash{\frac18}),
so it is guaranteed that
$\|K_t\|_\square \leq \smash{\tfrac 18}$.
Again because the loop does not terminate, both~$h$ and~$\wprec$ go to~0.
So~$T\to t$ in the Hausdorff metric, while~$A$, $x$ and~$r$ are fixed.
In particular, $K_T \to K_t$, by Proposition~\ref{prop:lipschitz-continuity}\ref{item:lipcont:dist}, which makes the two inequalities above contradict each other.

We now consider the iterations of the main loop.
Let~$x_k$, $r_k$, $A_k$, $t_k$, and~$T_k$ be the value of the respective variables at the end of the~$k$th iteration.
Moreover, let~$T_0 = \left\{ 0 \right\} \in \bx\mathbb{R}$ and let~$(x_0, r_0, A_0)$ denote the input Moore box.
Let~$N$ be the total number of iterations, perhaps infinite.
Recall that~$I$ is the maximal interval of definition of~$Z$.

\begin{lemma}\label{lem:track}
  For any~$0\leq k \leq N$,
  $T_k \subseteq I$ and
  for any~$t \in T_k$,
  $(x_k, r_k, A_k)$ is a~$\frac78$-Moore box for~$F_t$ with associated root~$Z_t$.
\end{lemma}

\begin{proof}
  We proceed by induction on~$k$.
  The base case~$k = 0$ is simply the input assumption.
  Assume the statement holds for~$k-1$.
  Let~$s$ be the supremum of~$T_{k-1}$, which is also the infimum of~$T_k$.
  By induction hypothesis,
  $(x_{k-1}, r_{k-1}, A_{k-1})$ is a~$\frac78$-Moore box for~$F_{s}$ with associated zero~$Z_{s}$.
  By definition,
  \[ (x_k, r_k, A_k) = \fn{refine}(F_{s}, x_{k-1}, r_{k-1}, A_{k-1}, \tfrac 18). \]
  (We now drop the index~$k$ everywhere.)
  The correctness of~\fn{refine} implies that~$(x, r, A)$ is a~$\smash{\frac18}$-Moore
  box of~$F_s$ with same associated zero.
  By construction,
  $\fn{M}(\bx F_T, \bx \ud F_T, x, r, A, \frac78)$ holds.
  By Lemma~\ref{lem:moore-interval}, this implies that
  $(x, r, A)$ is a~$\smash{\frac78}$-Moore box for~$F_t$ for any~$t \in T$.

  It remains to prove that~$T \subseteq I$ and that
  for any~$t\in T$, the zero of~$F_t$ associated to~$(x, r, A)$ is~$Z_t$.
  Let
  $J = \left\{ t \in T\cap I \st \| Z_t - x \| \leq r \right\}$.
  This is, by definition, a closed set in~$T\cap I$.
  It is not empty: $s \in J$.
  Moreover, for any~$t \in J$, $Z_t$ is the zero of~$F_t$ associated to~$(x, r, A)$, because there is a unique zero of~$F_t$ in the ball~$x + rB$.
  Therefore, by Theorem~\ref{thm:moore}\ref{item:rump:zetaloc}, we also have~$\|Z_t - x\| \leq \frac78 r$.
  It follows that~$J$ is also open in~$T\cap I$.
  Since~$T\cap I$, is an interval, we have~$J = T\cap I$.

  It only remains to prove that~$T \subseteq I$.
  Recall that~$s = \inf T \in I$ and let~$b = \sup (T\cap I)$.
  By definition,
  $\|Z_t - x\| \leq r$
  for any~$t\in T\cap I$,
  and it remains true as~$t\to b$.
  Moreover, by Theorem~\ref{thm:moore}\ref{item:rump:triple}, $\ud F_t(Z_t)^{-1}$ stays bounded,
  and it remains true as~$t\to b$.
  This shows that~$Z_t$ does not come close to a critical point of~$F_t$ as~$t\to b$.
  This prevents~$b$ from being a boundary point of~$I$ (see Section~\ref{sec:setting}).
  Therefore $T\subseteq I$.
\end{proof}

This proves a part of Theorem~\ref{thm:track}:
if Algorithm~\ref{algo:track} terminates,
then~$1 \in I$ (because~$1\in T_N \subseteq I$)
and~$Z_1$ is the associated zero of the output.
It remains to prove that if~$1 \in I$, then the algorithm terminates.
For contradiction, assume it does not.
The value of the variable~$t$ converges to some~$s \in [0,1]$.
The step size~$h$ goes to~0,
and~$T \to \left\{ s \right\}$ in the Hausdorff metric.
By construction, we also have~$\wprec \to 0$.
Lastly, Lemma~\ref{lem:track} and Theorem~\ref{thm:moore}\ref{item:rump:zetaloc} imply that~$x$ stays in a bounded set.
Indeed, we always have~$\|x - Z_t\| \leq r$ for some~$t\in [0,1]$, and~$r \leq 1$, by construction of \fn{refine}.
Since~$t\mapsto Z_t$ is continuous on~$[0,1]$, it is bounded.
Similarly Theorem~\ref{thm:moore}\ref{item:rump:triple2} implies that~$A$ stays in a bounded set,
because~$t\mapsto \ud F_t(Z_t)^{-1}$ is continuous for~$t\in [0,1]$.

By contruction of \fn{refine}, we always have
\begin{equation}
  \label{eq:4}
  \left \|- r^{-1} A \cdot \bx F_t(x) + (I_n - A \cdot \bx \ud F_t(x+ rB))\cdot B \right\|_\square \leq \tfrac18,
\end{equation}
after Line~\ref{line:track:refine} of each iteration.

We first consider the case where~$r$ stays away from~0:
there is some~$r_0 > 0$ such that~$r \geq r_0$ at every iteration.
Since~$h\to 0$, it is divided infinitely many times, and at least one~\fn{M}(\dotsc) check fails when this happens.
So infinitely often, we have
\begin{equation}\label{eq:19}
  \left\| - r^{-1} A \cdot \bx F_T(x) + (I_n - A\cdot \bx \ud F_T(x+ rB))\cdot B \right\|_\square > \tfrac12.
\end{equation}
As established above, $A$ and~$r$ are bounded, so we may assume that they converge (and~$r$ converges to a positive value).
Since~$t$ and~$T$ have the same limit, in the Hausdorff metric, and~$\wprec\to 0$, the two left-hand sides in~\eqref{eq:4} and~\eqref{eq:19}
also have the same limit (Propositon~\ref{prop:lipschitz-continuity}\ref{item:lipcont:dist}), which contradicts the inequalities.

So we assume that~$r$ does not stay away from zero,
and considering a subset of the iterations, we may assume that~$r \to 0$.
In this case, $x \to Z_s$.
By construction of \fn{refine}, we also have~$A \to \ud F_s(Z_s)^{-1}$.
Moreover, when some~$r < 1$ is returned, this means that the box could not be grown, that is
\[ \left\|- \tfrac12 r^{-1} A \cdot \bx F_t(x) + (I_n - A \cdot \bx \ud F_t(x+ 2rB))\cdot B \right\|_\square > \tfrac18. \]
The part $I_n - A \cdot \bx \ud F_t(x+ 2rB)$ converges to~0, and after multipliying by~2, this leads to
\begin{equation}
  \label{eq:21}
  \left\|- r^{-1} A \cdot \bx F_t(x) \right\|_\square \geq \tfrac14 + o(1),
\end{equation}
in contradiction with~\eqref{eq:4} which shows that this magnitude is at most~$\frac18 + o(1)$.
This concludes the proof of termination.


\section{Predictors}
\label{sec:predictors}

\subsection{Rationale}

The performance of Algorithm~\ref{algo:track}
can be greatly improved by incorporating Taylor models.
Consider an iteration of the main loop of \fn{track}.
After the \fn{refine} step, we have~$t$ $x$, $r$, and $A$ such that
\[    \left \|- r^{-1} A \cdot F_t(x) + (I_n - A \cdot \ud F_t(x+ rB))\cdot B \right\|\leq \tfrac18, \]
and we want a~$\delta > 0$ as large as possible such that for all~$\eta \in [0,\delta]$,
\[  \big \|- \underbrace{r^{-1} A \cdot F_{t+\eta}(x)}_{V} + \underbrace{(I_n - A \cdot \ud F_{t+\eta}(x+ rB))\cdot B}_{\Delta} \big\|\leq \tfrac78. \]
For an informal analysis, we may consider~$r$ and~$\delta$ as infinitesimally small and compute with first order expansions.
We obtain that
$\|\Delta \| = O(r + \eta)$
and
\[ \| V \| = r^{-1} \| v \| \eta + O(r^{-1} \eta^2) = O(r^{-1} \eta), \]
where~$v = -A \cdot \dot F_t(x)$ is the \emph{speed vector} related to variation of the zero that we track as the parameter changes (see Equation~\ref{eq:1}).
So the~$V$ term is likely to be the main obstruction in raising~$\delta$
and this suggests that we may expect~$\delta \simeq r$.

In the special case where the speed vector~$v$ vanishes, we may expect the much better~$\delta \simeq \sqrt{r}$.
In principle, it is easy to perform the first order correction
and reduce to the stationary case
by introducing the auxilliary system
$G_\eta(x) = F_{t+\eta}(x - \eta v)$
which is made to satisfy~$\dot G_0(x) = 0$.
If we can rigorously track a zero of~$G_\eta$
as~$\eta$ moves from~$0$,
we can certainly transfer this information to the original sytem~$F_{t+\eta}$.
Higher order corrections are also possible, we can enforce the cancellation of more terms in~$V$, which will increase the step size until the~$\Delta$ term takes over.
However, this idea does not combine nicely with interval arithmetic
which never cancels anything.
This phenomenom, well known as the \emph{dependency problem},
obliterates all ideas based on the cancelations of some dominant terms.
Taylor models is a classical way to circumvent it.

\subsection{Taylor models}

Let~$I \in \bx \mathbb{R}$ be an interval containing~$0$.
A \emph{Taylor model of order~$\nu$ on~$I$}
is a polynomial~$P(\eta) = a_0 + a_1 \eta + \dotsb + a_{\nu+1} \eta^{\nu+1}$ of degree at most~$r+1$ with coefficients in~$\bx \mathbb{C}$.
A Taylor model~$P$ on~$I$ encloses a function~$f : \mathbb{R} \to \mathbb{C}$
if for any~$t\in I$, there are~$\bar a_0,\dotsc,\bar a_{\nu+1} \in \mathbb{C}$ such that~$\bar a_i \in a_i$ and
$f(t) = \bar a_0 +\bar a_1 t+ \dotsb + \bar a_{\nu+1} t^{\nu+1}$.
If~$J \in \bx \mathbb{R}$ is a subinterval of~$I$
then the interval computation of~$P(J)$ contains~$f(J)$ for any function~$f$ enclosed by~$P$.

Since the variable~$t$ is bound to~$I$, we can squeeze a Taylor model of order~$r$ into one of order~$\nu-1$
by replacing the last two terms~$a_{\nu} \eta^{\nu} + a_{\nu+1} \eta^{\nu+1}$ by the single term~$\left( a_{\nu} \boxplus (a_{\nu+1} \boxtimes I) \right) \eta^{\nu}$.
If~$P$ encloses a function~$f$, then the reduced model still does.
We define on order-$\nu$ Taylor models
an addition by the componentwise addition of intervals~$\boxplus$.
We define also a multiplication by the usual polynomial multiplication formula, but using~$\boxplus$ and~$\boxtimes$, which leads to a Taylor model of order~$2\nu+1$,
followed by repeated sequeezing to reduce to order~$\nu$.
These operations are naturally compatible with the enclosure of functions.

In general, the first coefficients~$a_0,\dotsc,a_\nu$ of a Taylor model of order~$\nu$ are narrow intervals enclosing the Taylor expansion of an enclosed function,
their width reflect only roundoff errors.
The last term~$a_{\nu+1} \eta^{\nu+1}$ reflects the~$O(\eta^{\nu+1})$ term of an order-$\nu$ Taylor expansion.
For more details on Taylor models, we refer to \textcite{MooreKearfottCloud_2009,Neumaier_2003,Joldes_2011,BerzHoffstatter_1998}.

\subsection{Path tracking with a predictor}

We consider the same setting as in Section~\ref{sec:path-tracking}.
Suppose that for some~$t \in [0,1]$, we have a $\frac18$-Moore box $(x, r, A)$ for~$F_t$.
Suppose also that we have a vector~$\mathcal{X}(\eta)$ of polynomials
such that~$\mathcal{X}(0) = x$.
Naturally, we will choose~$\mathcal{X}(\eta)$ to approximate the zero~$Z_{t+\eta}$ of~$F_{t+\eta}$
the best we can, but we assume nothing.
$\mathcal{X}$ is called the \emph{predictor}.
Let also~$h > 0$ be what we think is a good step size.

Using the arithmetic of order-$\nu$ Taylor models on the domain~$[0, h]$ (we will typically choose~$\nu=3$), we compute
\[ \mathcal{K} = - r^{-1} A \cdot \bx F_{t+\eta}(\mathcal{X}) + \big[ I_n - A \cdot \bx \ud F_{t+\eta}(\mathcal{X} + rB) \big] \cdot B, \]
which is a vector of Taylor models.
Then we compute~$\mathcal{K}([0, h])$ and check if it is included in~$\frac78 B$.
If it is, then~$(\mathcal{X}(e), r, A)$ is a~$\smash{\frac 78}$-Moore box for~$F_{t+e}$ for any any~$e \in [0, h]$.
This follows from the compatibility of the Taylor model arithmetic with enclosures.

If~$\mathcal{K}([0, h])$ is not included in~$\frac 78 B$, we can try with~$\mathcal{K}([0, \frac h2])$,
or maybe~$\mathcal{K}([0, \smash{\frac h4}])$, we do not need to restart the computation of~$\mathcal{K}$ from scratch with a lower step size.
In principle, we may assume that~$\mathcal{K}(0) \subseteq \frac 18 B$, so
there should be some~$j \in \lparen 0,h \rbrack$ such that~$\mathcal{K}([0, j]) \subseteq \smash{\frac 78} B$,
but if we need~$j$ to be very small, it makes more sense instead to recompute~$\mathcal{K}$ over a smaller domain.

\begin{algorithm}[t]
  \raggedright
  \begin{description}
    \item[input]  $F_\bullet$, a circuit~$\mathbb{C}\times \mathbb{C}^{n} \to \mathbb{C}^n$;
          $(x, r, A)$, a $\frac78$-Moore box for~$F_0$
    \item[output] a Moore box for~$F_1$
  \end{description}
  \begin{pseudo}
    def \fn{track}(F_\bullet, x, r, A):\\+
      $t\gets 0$; \quad      $h \gets \frac12 $; \\
      while $t < 1$: \label{line:refine:main-loop}\\+
        $x, r, A \gets \fn{refine}(F_t, x, r, A, \frac 18)$ \label{line:track:refine}\\
        $h \gets \frac54 h$ \quad \ct{try growing the step size}\\
        $v \gets \md(- A \cdot \bx \dot F_{t}(x))$\\
        $\mathcal{X} \gets x + v \eta$ \quad \ct{$\eta$ is the variable of Taylor models}\\
        \ct{compute~$\mathcal{K}$ with order-2 Taylor model arithmetic on~$[0, h]$}\\
        $\mathcal K \gets -r^{-1} A \cdot \bx F_{t+\eta}(\mathcal{X}) + \big[I_n - A \cdot \bx \ud F_{t+\eta}(\mathcal{X} + rB)\big]\cdot B$\\
        if $\| \mathcal{K}([0, h]) \|_\square > \frac 78$:\\+
          $h \gets \frac h2$ \\
          $\wprec\gets \min(\wprec, h)$\\
          if $\| \mathcal{K}([0, \frac h2]) \|_\square > \frac 78$:\\+
             \ct{unsuccessful, restart the iteration with smaller~$h$}\\
             continue\\--

        $t \gets t + h$\\
        $x \gets \md( \mathcal{X}(h) )$\\-

    return $x, r, A$
  \end{pseudo}
  \caption{Path tracking with the tangent predictor}
  \label{algo:track:fast}
\end{algorithm}

What predictor can we choose?
Ideally, we would choose~$\mathcal{X}$ to be a truncated Taylor expansion of~$Z_{t+\eta}$ around~$\eta = 0$.
But only the first term is easy to get: by Equation~\eqref{eq:1}, we have
\[ Z_{t+\eta} = Z_t - \ud F_t(Z_t)^{-1} \cdot \dot F_t(Z_t) \eta + O(\eta^2). \]
Since~$x$ approximates~$Z_t$ and~$A$ approximates~$\ud F_t(Z_t)^{-1}$, we may choose
$\mathcal{X}_\text{tangent} = x - A \cdot \dot F_t(x) \eta$.
We compute~$\dot F_t(x)$ by automatic differentiation, similarly to~$\ud F_t(x)$.
This is the \emph{tangent predictor}
and it leads to Algorithm~\ref{algo:track:fast}.

\begin{theorem}
  Algorithm~\ref{algo:track:fast} is correct and terminates, in the sense of Theorem~\ref{thm:track}.
\end{theorem}

\emph{Mutatis mutandis}, the proof is the same as the one of Theorem~\ref{thm:track}.
The quality of the predictor does not matter much, as long as it stays bounded.

If we record the previous values of~$x$ and the tangent vector, and the previous step size,
we can compute the \emph{Hermite predictor}
\begin{multline*}
\mathcal{X}_\text{Hermite} = x + v \eta + (2 w - 3 \Delta x) \tfrac{\eta^2}{h_\text{prev}} + (w - 2 \Delta x) \tfrac{\eta^3}{h_\text{prev}^2},
\end{multline*}
where~$w = v + v_\text{prev}$ and~$\Delta x = h_\text{prev}^{-1} (x - x_\text{prev})$.
It is the unique polynomial with~$\mathcal{X}(0) = x$, $\mathcal{X}'(0) = v$, $\mathcal{X}(-h_\text{prev}) = x_\text{prev}$ and $\mathcal{X}'(-h_\text{prev}) = v_\text{prev}$.
This predictor, with order-3 Taylor models, gave us very good results, on which we report in Section~\ref{sec:experiments}.
Compared to uncertified method, the complexity of the predictor has a larger impact on the computation time: the predictor is not only used as a guess, we need to validate it.
So the balance between the complexity of the predictor and the number of iterations it saves does not favor high order methods.

By design of Algorithm~\ref{algo:track:fast}, from one Moore box $(x, r, A)$, at a given~$t$, to the next point~$x'$ at time~$t'$,
it is necessary that
$(I_n - A \cdot \ud F_{t'}(x'))\cdot B \subseteq \frac 78 B$.
This provides a theoretical maximum step length allowed by Moore's criterion,
independent of the quality of the predictor, or overestimation issues in interval arithmetic.
In experimentation, we observed that the Hermite predictor brings us very close to this theoretical limit.
So we do not expect that higher order predictors may improve performance.

\section{Experiments}
\label{sec:experiments}

\subsection{Implementation}

We propose a Rust implementation of Algorithm~\ref{algo:track:fast}, using the Hermite predictor.
Only fixed precision is implemented: all intervals have 64-bits floating-point endpoints,
and when the algorithm warns about precision, the computation is aborted.
Interval arithmetic is implemented using the AVX instruction set for the x86-64 platform,
following \textcite{Lambov_2008}. For example, the multiplication of two order-3 Taylor models,
that is 100 real interval number multiplications and 40 additions,
is performed with 638~SIMD instructions in less than 300 CPU cycles, according to the analysis tool \emph{llvm-mca}.
The source code is distributed under the GPLv3 license and available at
\begin{center}
  {\sffamily \url{https://gitlab.inria.fr/numag/algpath}}.
\end{center}

\subsection{Timings}

\begin{table*}[t]
  \centering
  \footnotesize
  \setlength{\tabcolsep}{2pt}\begin{tabular*}{\textwidth}{l@{\extracolsep{\fill}}rrrrrrrrrrrrrrrrrrrr}
\toprule
 &  &  &  & \multicolumn{2}{c}{circuit size} & \multicolumn{5}{c}{HomotopyContinuation.jl} & \multicolumn{5}{c}{algpath} & \multicolumn{5}{c}{Macaulay2} \\
    \cmidrule(lr){5-6} \cmidrule(lr){7-11} \cmidrule(lr){12-16} \cmidrule(lr){17-21}
\multicolumn{1}{c}{name} & \multicolumn{1}{c}{dim.} & \multicolumn{1}{c}{max deg} & \multicolumn{1}{c}{\# paths} & \multicolumn{1}{c}{$f$} & \multicolumn{1}{c}{$\mathrm{d}f$} & \multicolumn{1}{c}{fail.} & \multicolumn{1}{c}{med.} & \multicolumn{1}{c}{max.} & \multicolumn{1}{c}{ksteps/s } & \multicolumn{1}{c}{time (s)} & \multicolumn{1}{c}{fail.} & \multicolumn{1}{c}{med.} & \multicolumn{1}{c}{max.} & \multicolumn{1}{c}{ksteps/s } & \multicolumn{1}{c}{time (s)} & \multicolumn{1}{c}{fail.} & \multicolumn{1}{c}{med.} & \multicolumn{1}{c}{max.} & \multicolumn{1}{c}{ksteps/s } & \multicolumn{1}{c}{time (s)} \\
\midrule
dense & 1 & 10 & 10 & 88 & 96 &  & 6 & 10 & 30 & 1.8 &  & 11 & 31 & 55 & $<$ 0.1 &  & 629 & 2146 & 55 & 0.2 \\
dense & 1 & 20 & 20 & 168 & 202 &  & 13 & 23 & 53 & 1.8 &  & 29 & 134 & 42 & $<$ 0.1 &  & 40\,k & 183\,k & 46 & 20 \\
dense & 1 & 30 & 30 & 248 & 314 &  & 10 & 25 & 41 & 2.0 &  & 23 & 372 & 25 & $<$ 0.1 &  & 830\,k & 3478\,k & 30 & 18 min \\
dense & 1 & 40 & 40 & 328 & 416 &  & 14 & 30 & 45 & 2.0 &  & 34 & 197 & 24 & $<$ 0.1 & \multicolumn{5}{c}{$>$ 1 h} \\
dense & 1 & 50 & 50 & 408 & 520 &  & 12 & 61 & 37 & 1.9 &  & 30 & 5567 & 13 & 0.7 & \multicolumn{5}{c}{$>$ 1 h} \\
dense & 1 & 100 & 100 & 808 & 1054 &  & 13 & 51 & 23 & 1.9 &  & 38 & 5289 & 7.4 & 1.4 & \multicolumn{5}{c}{$>$ 1 h} \\
dense & 1 & 500 & 500 & 4008 & 5466 &  & 14 & 59 & 3.8 & 3.9 & 2 & 60 & 1121 & 2.3 & 17 & 500 &  &  &  & 4.0 \\
dense & 1 & 1000 & 1000 & 8008 & 10952 &  & 15 & 100 & 1.7 & 12 & 35 & 74 & 976 & 1.1 & 82 & 1000 &  &  &  & 29 \\
dense & 2 & 5 & 25 & 316 & 368 &  & 23 & 56 & 57 & 2.3 &  & 50 & 95 & 25 & $<$ 0.1 &  & 2850 & 6736 & 53 & 1.4 \\
dense & 2 & 10 & 100 & 1016 & 1280 &  & 22 & 74 & 33 & 2.6 &  & 53 & 307 & 9.2 & 0.7 &  & 33\,k & 301\,k & 28 & 158 \\
dense & 2 & 20 & 400 & 3616 & 4612 &  & 25 & 63 & 13 & 3.1 &  & 74 & 401 & 2.9 & 12 & \multicolumn{5}{c}{$>$ 1 h} \\
dense & 2 & 30 & 900 & 7816 & 9952 &  & 24 & 127 & 5.8 & 6.4 &  & 85 & 690 & 1.4 & 72 & \multicolumn{5}{c}{$>$ 1 h} \\
dense & 2 & 40 & 1600 & 13616 & 17284 &  & 25 & 95 & 3.4 & 14 &  & 100 & 998 & 0.81 & 268 & \multicolumn{5}{c}{$>$ 1 h} \\
dense & 2 & 50 & 2500 & 21016 & 26624 &  & 27 & 84 & 2.3 & 33 &  & 117 & 1675 & 0.53 & 12 min & \multicolumn{5}{c}{$>$ 1 h} \\
katsura & 5 & 2 & 16 & 192 & 98 &  & 49 & 74 & 41 & 3.8 &  & 74 & 136 & 26 & $<$ 0.1 &  & 3833 & 7903 & 38 & 1.9 \\
katsura & 7 & 2 & 64 & 310 & 158 &  & 59 & 99 & 59 & 3.9 &  & 100 & 203 & 15 & 0.5 &  & 5963 & 15\,k & 26 & 16 \\
katsura & 9 & 2 & 256 & 448 & 228 &  & 82 & 132 & 54 & 4.2 &  & 148 & 286 & 9.5 & 4.2 &  & 12\,k & 59\,k & 18 & 186 \\
katsura & 11 & 2 & 1024 & 606 & 308 &  & 100 & 179 & 41 & 6.7 &  & 177 & 359 & 6.3 & 30 &  & 21\,k & 88\,k & 13 & 30 min \\
katsura & 16 & 2 & 32768 & 1090 & 548 &  & 153 & 303 & 22 & 235 &  & 304 & 1847 & 2.7 & 1 h & \multicolumn{5}{c}{$>$ 50 h} \\
katsura & 21 & 2 & 1048576 & 1696 & 844 &  & 209 & 469 & 13 & 4 h & 483 & 427 & 8798 & 1.4 & 101 h & \multicolumn{5}{c}{not benchmarked} \\
katsura * & 26 & 2 & 100 & 2430 & 1202 &  & 305 & 466 & 6.9 & 8.8 & 1 & 800 & 2930 & 0.73 & 125 & \multicolumn{5}{c}{$>$ 1 h} \\
katsura * & 31 & 2 & 100 & 3286 & 1614 &  & 382 & 538 & 4.9 & 12 & 1 & 852 & 5021 & 0.47 & 219 & \multicolumn{5}{c}{$>$ 1 h} \\
katsura * & 41 & 2 & 100 & 5376 & 2618 &  & 554 & 787 & 2.7 & 24 & 9 & 1371 & 5182 & 0.19 & 13 min & \multicolumn{5}{c}{$>$ 1 h} \\
dense * & 4 & 3 & 100 & 1080 & 1318 &  & 39 & 67 & 41 & 2.4 &  & 66 & 127 & 8.3 & 0.9 &  & 3384 & 9936 & 35 & 10 \\
dense * & 6 & 3 & 100 & 4092 & 5384 &  & 54 & 96 & 9.0 & 3.3 &  & 112 & 224 & 2.3 & 5.1 &  & 11\,k & 24\,k & 18 & 62 \\
dense * & 8 & 3 & 100 & 11120 & 15242 &  & 73 & 124 & 2.1 & 6.3 &  & 157 & 354 & 0.86 & 19 &  & 21\,k & 74\,k & 9.5 & 243 \\
structured * & 4 & 3 & 100 & 244 & 418 &  & 40 & 78 & 92 & 4.0 &  & 75 & 199 & 24 & 0.4 &  & 4531 & 8925 & 41 & 11 \\
structured * & 6 & 3 & 100 & 426 & 778 &  & 66 & 101 & 59 & 3.9 &  & 130 & 254 & 13 & 1.1 &  & 18\,k & 61\,k & 23 & 85 \\
structured * & 8 & 3 & 100 & 670 & 1252 &  & 81 & 121 & 40 & 3.9 &  & 182 & 283 & 7.9 & 2.3 &  & 36\,k & 97\,k & 13 & 305 \\
structured \textsuperscript{N} & 5 & 5 & 1 & 302 & 545 &  & 42 & 42 & 4.9 & 3.1 &  & 99 & 99 & 18 & $<$ 0.1 &  & 252\,k & 252\,k & 12 & 22 \\
structured \textsuperscript{N} & 10 & 10 & 1 & 1034 & 2024 &  & 53 & 53 & 0.18 & 3.1 &  & 123 & 123 & 4.9 & $<$ 0.1 & \multicolumn{5}{c}{$>$ 1 h} \\
structured \textsuperscript{N} & 15 & 15 & 1 & 2366 & 5079 & \multicolumn{5}{c}{$>$ 8 GB} &  & 628 & 628 & 2.0 & 0.4 & \multicolumn{5}{c}{$>$ 8 GB} \\
structured \textsuperscript{N} & 20 & 20 & 1 & 3554 & 6721 & \multicolumn{5}{c}{$>$ 8 GB} &  & 1591 & 1591 & 1.2 & 1.5 & \multicolumn{5}{c}{$>$ 8 GB} \\
structured \textsuperscript{N} & 25 & 25 & 1 & 5466 & 10541 & \multicolumn{5}{c}{$>$ 8 GB} &  & 1734 & 1734 & 0.69 & 2.9 & \multicolumn{5}{c}{$>$ 8 GB} \\
structured \textsuperscript{N} & 30 & 30 & 1 & 7788 & 15239 & \multicolumn{5}{c}{$>$ 8 GB} &  & 1989 & 1989 & 0.43 & 5.2 & \multicolumn{5}{c}{$>$ 8 GB} \\
dense \textsuperscript{N} & 4 & 3 & 1 & 792 & 1038 &  & 18 & 18 & 0.71 & 2.5 &  & 50 & 50 & 12 & $<$ 0.1 &  & 21\,k & 21\,k & 27 & 0.8 \\
dense \textsuperscript{N} & 6 & 3 & 1 & 3072 & 4376 &  & 33 & 33 & 0.12 & 2.7 &  & 90 & 90 & 3.6 & 0.1 &  & 22\,k & 22\,k & 13 & 1.8 \\
dense \textsuperscript{N} & 8 & 3 & 1 & 8464 & 12602 &  & 10 & 10 & $<$ 0.01 & 5.0 &  & 35 & 35 & 1.1 & 0.4 &  & 8775 & 8775 & 5.5 & 1.6 \\
structured \textsuperscript{N} & 4 & 3 & 1 & 200 & 296 &  & 32 & 32 & 4.9 & 3.2 &  & 66 & 66 & 28 & $<$ 0.1 &  & 8559 & 8559 & 34 & 0.3 \\
structured \textsuperscript{N} & 6 & 3 & 1 & 350 & 516 &  & 32 & 32 & 2.9 & 2.9 &  & 79 & 79 & 15 & $<$ 0.1 &  & 31\,k & 31\,k & 19 & 1.7 \\
structured \textsuperscript{N} & 8 & 3 & 1 & 566 & 876 &  & 21 & 21 & 1.0 & 2.9 &  & 73 & 73 & 8.8 & $<$ 0.1 &  & 19\,k & 19\,k & 8.0 & 2.3 \\
\bottomrule
\end{tabular*}

  \bigskip
  \caption{Comparison of \emph{HomotopyContinuation.jl}, \emph{algpath} (this work) and \emph{Macaulay2}.
    \emph{Dim.}: number of variables;
    \emph{deg.}: maximum degree of the equations;
    \emph{\#~paths}: number of paths to track;
    \emph{circuit size}: size of the circuit evaluating the parametric system and its derivative;
    \emph{fail.}: number of reported failures;
    \emph{med.}: median number of iterations over all paths;
    \emph{max.}: maximum number of iterations in a single path;
    \emph{ksteps/s}: number of steps per second (thousands);
    \emph{time}: total time to track all the paths.\quad
    *\enskip A hundred randomly picked starting zeros of a total degree homotopy.\quad
    \textsuperscript{N}\enskip Newton homotopy.}
  \vspace{-.5cm}
  \label{table:bench}
\end{table*}

We compared our implementation \emph{algpath} with the Macaulay2 package \emph{NumericalAlgebraicGeometry} \parencite{BeltranLeykin_2012}.
As a state-of-the-art implementation of noncertified path tracking,
we also benchmarked the Julia package \emph{HomotopyContinuation.jl} \parencite{BreidingTimme_2018}.
We used a computer with an Intel Xeon E3-1220v3 CPU and 16GB of RAM. Path tracking algorithms uses very little memory, but we report some out-of-memory errors from Julia (to be investigated)
and Macaulay2 (which tries to expand structured polynomials).
The benchmarking data, scripts and raw results are available at
\begin{center}
  {\sffamily\url{https://gitlab.inria.fr/numag/algpath-bench}}.
\end{center}

\subsubsection{Data set}
We considered linear homotopies, $F_t(x) = tf(x) + (1 - t)g (x)$ between a start polynomial system $g$ and a target system $f$.
For the target, we considered several families. First, dense systems with random standard independent complex coefficients.
Second, structured random systems, with low Waring rank, with components of the form $\pm 1 + \smash{\sum_{i = 1}^5 \big( \sum_{j = 1}^n a_{i, j}x_j \big)^d}$
with random coefficients $a_{i, j}$, independent and uniformly distributed in $\{-1, 0, 1\}$.
Third, we considered the classical benchmark family Katsura~$n$ (available in Sagemath with~\texttt{sage.rings.ideal.Katsura}),
which is a polynomial system in~$n+1$ variables with~$2^n$ solutions.

For the start system, we considered total degree homotopies, with~$g_i(x) = \gamma_i (x_i^d - 1)$
and~$\gamma_i \in \mathbb{C}$ random,
as well as Newton homotopies, with~$g(x) = f(x) - f(x_0)$ for some random~$x_0$.

\subsubsection{Results analysis}
Table~\ref{table:bench} shows that the number of steps per second performed by \emph{Macaulay2} is comparable to that of \emph{HomotopyContinuation.jl}\footnotemark. By strongly decreasing the number of steps required to track a path compared to \emph{Macaulay2}, this work is able to solve problems of a much bigger scale. The gap between certified and noncertified methods is significantly narrowed.
\footnotetext{Rigorous benchmarking of Julia code is difficult because of run-time compilation (JIT).
Following common practice, we run twice \texttt{HomopotopyContinuation.solve} with exactly the same arguments.
The first run suffers from compilation overheads, while the second does not.
However, some compilation overheads are input dependent.
For fair comparison, it makes no sense to only time the second run (because it is too easy to compute something faster after we did it once).
So the \emph{total time} in Table~\ref{table:bench} is the time of the first run. To lessen the compilation overheads in this first run,
we perform a warmup run with a polynomial system of degree~1 and same dimension as the input system.
In constrast, the \emph{number of steps per second} is obtained from the time of the second run, divided by the total number of iterations.}

Comparing the median number of steps performed by \emph{HC.jl} and \emph{algpath}
on each problem suggests that, typically, the latter performs only 2 or~3 times more steps.
Figure~\ref{fig:steps} inspects this relation more precisely.
The correlation between the number of steps of \emph{HC.jl} and \emph{algpath} is strong on some examples
(such as random dense or structured dimension-8 degree-3 systems), but weaker on some others (difficult paths in the degree~500 univariate polynomial, or Katsura examples).

Finally, we reach the limits of 64-bits floating-point arithmetic sooner than \emph{HomotopyContinuation.jl}, as shown by the number of failures for high degree univariate instances, or Katsura-40.
Note that we did not check systematically the absence of path swapping in \emph{HC.jl}, there may be silent failures.

\begin{acks}
  We are grateful to Florent Bréhard, and Mioara Joldes for
  useful discussions, and to the referees for thoughtful reports.
  This work has been supported by the \grantsponsor{anr}{Agence nationale de la
    recher\-che (ANR)}{https://anr.fr}, grant agreement
  \grantnum{anr}{ANR-19-CE40-0018} (De Rerum Natura); and by the
  \grantsponsor{erc}{European Research Council (ERC)}{https://erc.europa.eu}
  under the European Union’s Horizon Europe research and innovation programme,
  grant agreement \grantnum{erc}{101040794} (10000~DIGITS).
\end{acks}

\begin{figure*}[tb]
  \def\markpenta{\tikz \draw[fill={rgb,1:red,0.1725;green,0.6275;blue,0.1725}] plot[mark=pentagon*] coordinates {(0,0)};}
  \def\markutriangle{\tikz \draw[fill={rgb,1:red,0.1216;green,0.4667;blue,0.7059}] plot[mark=triangle*] coordinates {(0,0)};}
  \def\markcircle{\tikz \draw[fill={rgb,1:red,0.6824;green,0.7804;blue,0.9098}] plot[mark=*] coordinates {(0,0)};}
  \def\markdtriangle{\tikz \draw[fill={rgb,1:red,1.0;green,0.7333;blue,0.4706}] plot[mark=triangle*, mark options={rotate=180}] coordinates {(0,0)};}
  \def\markdiamond{\tikz \draw[fill={rgb,1:red,1.0;green,0.498;blue,0.0549}] plot[mark=diamond*] coordinates {(0,0)};}
  \centering
  \input{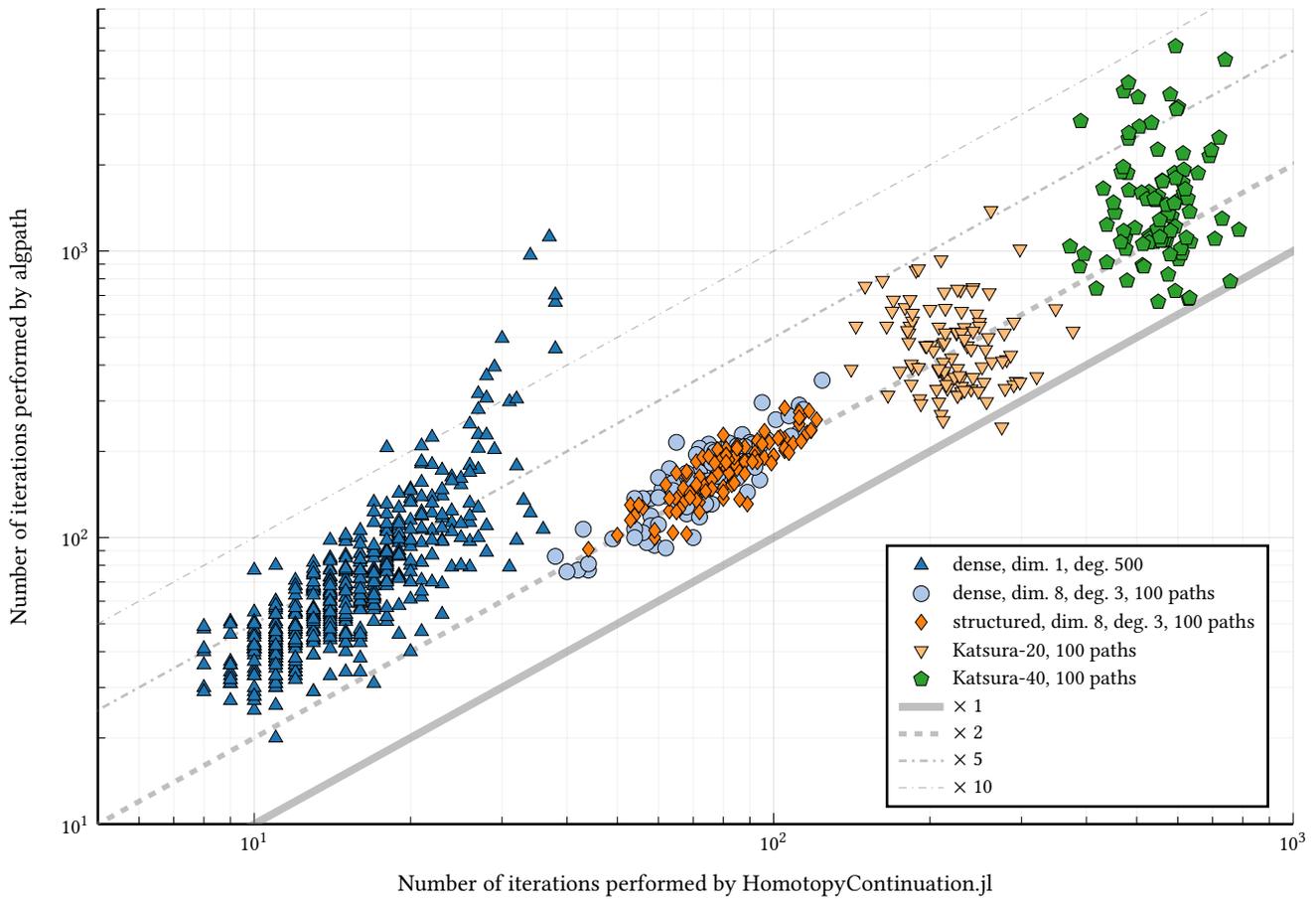}
  \vspace{-.7cm}
  \caption{Number of iterations performed by \emph{algpath} (this work) and \emph{HomotopyContinuation.jl} (noncertified path tracking)
    in four path tracking problems. We observe that \emph{algpath} performs typically no more than 5 times more iterations
    than a state-of-the-art noncertified numerical solver. The ratio is close to~2 on well-conditioned examples (\protect\markcircle\ and \protect\markdiamond)
    but there is much more variability on poor conditioning (\protect\markutriangle, \protect\markdtriangle\  and \protect\markpenta).
  }
  \label{fig:steps}
\end{figure*}



\begin{thebibliography}{33}


\ifx \showCODEN    \undefined \def \showCODEN     #1{\unskip}     \fi
\ifx \showDOI      \undefined \def \showDOI       #1{#1}\fi
\ifx \showISBNx    \undefined \def \showISBNx     #1{\unskip}     \fi
\ifx \showISBNxiii \undefined \def \showISBNxiii  #1{\unskip}     \fi
\ifx \showISSN     \undefined \def \showISSN      #1{\unskip}     \fi
\ifx \showLCCN     \undefined \def \showLCCN      #1{\unskip}     \fi
\ifx \shownote     \undefined \def \shownote      #1{#1}          \fi
\ifx \showarticletitle \undefined \def \showarticletitle #1{#1}   \fi
\ifx \showURL      \undefined \def \showURL       {\relax}        \fi
\providecommand\bibfield[2]{#2}
\providecommand\bibinfo[2]{#2}
\providecommand\natexlab[1]{#1}
\providecommand\showeprint[2][]{arXiv:#2}

\bibitem[Alefeld and Mayer(2000)]%
        {AlefeldMayer_2000}
\bibfield{author}{\bibinfo{person}{G.~Alefeld} {and}
  \bibinfo{person}{G.~Mayer}.} \bibinfo{year}{2000}\natexlab{}.
\newblock \showarticletitle{Interval Analysis: Theory and Applications}.
\newblock \bibinfo{journal}{\emph{J. Comput. Appl. Math.}}
  \bibinfo{volume}{121}, \bibinfo{number}{1} (\bibinfo{year}{2000}),
  \bibinfo{pages}{421--464}.
\newblock


\bibitem[Bates et~al\mbox{.}(2013)]%
        {BatesHauensteinSommeseWampler_2013}
\bibfield{author}{\bibinfo{person}{D.~J. Bates}, \bibinfo{person}{J.~D.
  Hauenstein}, \bibinfo{person}{A.~J. Sommese}, {and} \bibinfo{person}{C.~W.
  Wampler}.} \bibinfo{year}{2013}\natexlab{}.
\newblock \bibinfo{booktitle}{\emph{Numerically Solving Polynomial Systems with
  {{Bertini}}}}. \bibinfo{series}{Software, {{Environments}}, and {{Tools}}},
  Vol.~\bibinfo{volume}{25}.
\newblock \bibinfo{publisher}{SIAM, Philadelphia, PA}.
\newblock
\showISBNx{978-1-61197-269-6}


\bibitem[Beltr{\'a}n and Leykin(2012)]%
        {BeltranLeykin_2012}
\bibfield{author}{\bibinfo{person}{C.~Beltr{\'a}n} {and}
  \bibinfo{person}{A.~Leykin}.} \bibinfo{year}{2012}\natexlab{}.
\newblock \showarticletitle{Certified Numerical Homotopy Tracking}.
\newblock \bibinfo{journal}{\emph{Exp. Math.}} \bibinfo{volume}{21},
  \bibinfo{number}{1} (\bibinfo{year}{2012}), \bibinfo{pages}{69--83}.
\newblock
\showISSN{1058-6458}
\urldef\tempurl%
\url{https://doi.org/10/ggck73}
\showDOI{\tempurl}


\bibitem[Beltr{\'a}n and Leykin(2013)]%
        {BeltranLeykin_2013}
\bibfield{author}{\bibinfo{person}{C.~Beltr{\'a}n} {and}
  \bibinfo{person}{A.~Leykin}.} \bibinfo{year}{2013}\natexlab{}.
\newblock \showarticletitle{Robust Certified Numerical Homotopy Tracking}.
\newblock \bibinfo{journal}{\emph{Found. Comput. Math.}} \bibinfo{volume}{13},
  \bibinfo{number}{2} (\bibinfo{year}{2013}), \bibinfo{pages}{253--295}.
\newblock
\showISSN{1615-3383}
\urldef\tempurl%
\url{https://doi.org/10/ggck74}
\showDOI{\tempurl}


\bibitem[Berz and Hoffst{\"a}tter(1998)]%
        {BerzHoffstatter_1998}
\bibfield{author}{\bibinfo{person}{M.~Berz} {and}
  \bibinfo{person}{G.~Hoffst{\"a}tter}.} \bibinfo{year}{1998}\natexlab{}.
\newblock \showarticletitle{Computation and Application of {{Taylor}}
  Polynomials with Interval Remainder Bounds}.
\newblock \bibinfo{journal}{\emph{Reliab. Comput.}} \bibinfo{volume}{4},
  \bibinfo{number}{1} (\bibinfo{year}{1998}), \bibinfo{pages}{83--97}.
\newblock
\showISSN{1573-1340}
\urldef\tempurl%
\url{https://doi.org/bqmsdm}
\showDOI{\tempurl}


\bibitem[Breiding and Timme(2018)]%
        {BreidingTimme_2018}
\bibfield{author}{\bibinfo{person}{P.~Breiding} {and}
  \bibinfo{person}{S.~Timme}.} \bibinfo{year}{2018}\natexlab{}.
\newblock \showarticletitle{{{HomotopyContinuation}}.Jl: {{A}} Package for
  Homotopy Continuation in Julia}. In \bibinfo{booktitle}{\emph{Int. {{Congr}}.
  {{Math}}. {{Softw}}.}} \bibinfo{pages}{458--465}.
\newblock
\urldef\tempurl%
\url{https://doi.org/10/ggck7q}
\showDOI{\tempurl}


\bibitem[B{\"u}rgisser(2000)]%
        {Burgisser_2000}
\bibfield{author}{\bibinfo{person}{P.~B{\"u}rgisser}.}
  \bibinfo{year}{2000}\natexlab{}.
\newblock \bibinfo{booktitle}{\emph{Completeness and Reduction in Algebraic
  Complexity Theory}}.
\newblock \bibinfo{publisher}{Springer-Verlag}.
\newblock
\showISBNx{978-3-540-66752-0}
\urldef\tempurl%
\url{https://doi.org/10/d9n4}
\showDOI{\tempurl}


\bibitem[Cucker(2021)]%
        {Cucker_2021}
\bibfield{author}{\bibinfo{person}{F.~Cucker}.}
  \bibinfo{year}{2021}\natexlab{}.
\newblock \showarticletitle{Smale 17th Problem: Advances and Open Directions}.
\newblock \bibinfo{journal}{\emph{N. Z. J. Math.}}  \bibinfo{volume}{52}
  (\bibinfo{year}{2021}), \bibinfo{pages}{233--257}.
\newblock
\showISSN{1179-4984}
\urldef\tempurl%
\url{https://doi.org/gtc2rc}
\showDOI{\tempurl}


\bibitem[Duff and Lee(2024)]%
        {DuffLee_2024}
\bibfield{author}{\bibinfo{person}{T.~Duff} {and} \bibinfo{person}{K.~Lee}.}
  \bibinfo{year}{2024}\natexlab{}.
\newblock \bibinfo{title}{Certified Homotopy Tracking Using the {{Krawczyk}}
  Method}.
\newblock
\newblock
\showeprint[arxiv]{2402.07053}


\bibitem[Hauenstein et~al\mbox{.}(2014)]%
        {HauensteinHaywoodLiddell_2014}
\bibfield{author}{\bibinfo{person}{J.~D. Hauenstein},
  \bibinfo{person}{I.~Haywood}, {and} \bibinfo{person}{A.~C. Liddell, Jr.}}
  \bibinfo{year}{2014}\natexlab{}.
\newblock \showarticletitle{An a Posteriori Certification Algorithm for
  {{Newton}} Homotopies}. In \bibinfo{booktitle}{\emph{Proc. {{ISSAC}} 2014}}. \bibinfo{publisher}{ACM},
  \bibinfo{pages}{248--255}.
\newblock
\showISBNx{978-1-4503-2501-1}
\urldef\tempurl%
\url{https://doi.org/10/ggck7h}
\showDOI{\tempurl}


\bibitem[Hauenstein and Liddell(2016)]%
        {HauensteinLiddell_2016}
\bibfield{author}{\bibinfo{person}{J.~D. Hauenstein} {and}
  \bibinfo{person}{A.~C. Liddell}.} \bibinfo{year}{2016}\natexlab{}.
\newblock \showarticletitle{Certified Predictor--Corrector Tracking for
  {{Newton}} Homotopies}.
\newblock \bibinfo{journal}{\emph{J. Symb. Comput.}}  \bibinfo{volume}{74}
  (\bibinfo{year}{2016}), \bibinfo{pages}{239--254}.
\newblock
\showISSN{0747-7171}
\urldef\tempurl%
\url{https://doi.org/10/ggck7j}
\showDOI{\tempurl}


\bibitem[Hauenstein et~al\mbox{.}(2017)]%
        {HauensteinRodriguezSottile_2017}
\bibfield{author}{\bibinfo{person}{J.~D. Hauenstein}, \bibinfo{person}{J.~I.
  Rodriguez}, {and} \bibinfo{person}{F.~Sottile}.}
  \bibinfo{year}{2017}\natexlab{}.
\newblock \showarticletitle{Numerical Computation of {{Galois}} Groups}.
\newblock \bibinfo{journal}{\emph{Found. Comput. Math.}}
  (\bibinfo{year}{2017}), \bibinfo{pages}{1--24}.
\newblock
\showISSN{1615-3375, 1615-3383}
\urldef\tempurl%
\url{https://doi.org/gd2rw6}
\showDOI{\tempurl}


\bibitem[Johansson(2017)]%
        {Johansson_2017}
\bibfield{author}{\bibinfo{person}{F.~Johansson}.}
  \bibinfo{year}{2017}\natexlab{}.
\newblock \showarticletitle{Arb: Efficient Arbitrary-Precision Midpoint-Radius
  Interval Arithmetic}.
\newblock \bibinfo{journal}{\emph{IEEE Trans. Comput.}} \bibinfo{volume}{66},
  \bibinfo{number}{8} (\bibinfo{year}{2017}), \bibinfo{pages}{1281--1292}.
\newblock
\urldef\tempurl%
\url{https://doi.org/10/gbn9sm}
\showDOI{\tempurl}


\bibitem[Joldes(2011)]%
        {Joldes_2011}
\bibfield{author}{\bibinfo{person}{M.~Joldes}.}
  \bibinfo{year}{2011}\natexlab{}.
\newblock \emph{\bibinfo{title}{Rigorous Polynomial Approximations and
  Applications}}.
\newblock \bibinfo{thesistype}{Ph.\,D. Dissertation}.
  \bibinfo{school}{{\'E}cole normale sup{\'e}rieure de lyon}.
\newblock
\urldef\tempurl%
\url{https://theses.hal.science/tel-00657843}
\showURL{%
\tempurl}


\bibitem[Kranich(2015)]%
        {Kranich_2015}
\bibfield{author}{\bibinfo{person}{S.~Kranich}.}
  \bibinfo{year}{2015}\natexlab{}.
\newblock \bibinfo{title}{An Epsilon-Delta Bound for Plane Algebraic Curves and
  Its Use for Certified Homotopy Continuation of Systems of Plane Algebraic
  Curves}.
\newblock
\newblock
\showeprint[arxiv]{1505.03432}


\bibitem[Krawczyk(1969)]%
        {Krawczyk_1969}
\bibfield{author}{\bibinfo{person}{R.~Krawczyk}.}
  \bibinfo{year}{1969}\natexlab{}.
\newblock \showarticletitle{{Newton-Algorithmen zur Bestimmung von Nullstellen
  mit Fehlerschranken}}.
\newblock \bibinfo{journal}{\emph{Computing}} \bibinfo{volume}{4},
  \bibinfo{number}{3} (\bibinfo{year}{1969}), \bibinfo{pages}{187--201}.
\newblock
\showISSN{1436-5057}
\urldef\tempurl%
\url{https://doi.org/10/css7z9}
\showDOI{\tempurl}


\bibitem[Lambov(2008)]%
        {Lambov_2008}
\bibfield{author}{\bibinfo{person}{B.~Lambov}.}
  \bibinfo{year}{2008}\natexlab{}.
\newblock \showarticletitle{Interval {{Arithmetic Using SSE-2}}}. In
  \bibinfo{booktitle}{\emph{Reliab. {{Implement}}. {{Real Number Algorithms}}}}
  \emph{(\bibinfo{series}{Lecture {{Notes}} in {{Computer Science}}})},
  \bibfield{editor}{\bibinfo{person}{P.~Hertling},
  \bibinfo{person}{C.~M. Hoffmann}, \bibinfo{person}{W.~Luther},
  {and} \bibinfo{person}{N.~Revol}} (Eds.).
  \bibinfo{publisher}{Springer},
  \bibinfo{pages}{102--113}.
\newblock
\showISBNx{978-3-540-85521-7}
\urldef\tempurl%
\url{https://doi.org/c7vvrk}
\showDOI{\tempurl}


\bibitem[Lang(1997)]%
        {Lang_1997}
\bibfield{author}{\bibinfo{person}{S.~Lang}.} \bibinfo{year}{1997}\natexlab{}.
\newblock \bibinfo{booktitle}{\emph{Undergraduate Analysis}
  (\bibinfo{edition}{2} ed.)}.
\newblock \bibinfo{publisher}{Springer}.
\newblock
\showISBNx{978-1-4419-2853-5 978-1-4757-2698-5}
\urldef\tempurl%
\url{https://doi.org/gtcznr}
\showDOI{\tempurl}


\bibitem[Leykin(2011)]%
        {Leykin_2011}
\bibfield{author}{\bibinfo{person}{A.~Leykin}.}
  \bibinfo{year}{2011}\natexlab{}.
\newblock \showarticletitle{Numerical Algebraic Geometry}.
\newblock \bibinfo{journal}{\emph{J. Softw. Algebra Geom.}}
  \bibinfo{volume}{3}, \bibinfo{number}{1} (\bibinfo{year}{2011}),
  \bibinfo{pages}{5--10}.
\newblock
\showISSN{1948-7916, 1948-7916}
\urldef\tempurl%
\url{https://doi.org/ggck9r}
\showDOI{\tempurl}


\bibitem[{Marco-Buzunariz} and Rodr{\'i}guez(2016)]%
        {MarcoBuzunarizRodriguez_2016}
\bibfield{author}{\bibinfo{person}{M.~{\'A}. {Marco-Buzunariz}} {and}
  \bibinfo{person}{M.~Rodr{\'i}guez}.} \bibinfo{year}{2016}\natexlab{}.
\newblock \showarticletitle{{{SIROCCO}}: {{A}} Library for Certified Polynomial
  Root Continuation}. In \bibinfo{booktitle}{\emph{Proc. {{ICMS}} 2016}}
  \emph{(\bibinfo{series}{{{LNCS}}})}. \bibinfo{publisher}{Springer},
  \bibinfo{pages}{191--197}.
\newblock
\showISBNx{978-3-319-42432-3}
\urldef\tempurl%
\url{https://doi.org/10/grqk32}
\showDOI{\tempurl}


\bibitem[Moore(1977)]%
        {Moore_1977}
\bibfield{author}{\bibinfo{person}{R.~E. Moore}.}
  \bibinfo{year}{1977}\natexlab{}.
\newblock \showarticletitle{A Test for Existence of Solutions to Nonlinear
  Systems}.
\newblock \bibinfo{journal}{\emph{SIAM J. Numer. Anal.}} \bibinfo{volume}{14},
  \bibinfo{number}{4} (\bibinfo{year}{1977}), \bibinfo{pages}{611--615}.
\newblock
\showISSN{0036-1429}
\urldef\tempurl%
\url{https://doi.org/10/c66n76}
\showDOI{\tempurl}
\showeprint[jstor]{2156481}


\bibitem[Moore et~al\mbox{.}(2009)]%
        {MooreKearfottCloud_2009}
\bibfield{author}{\bibinfo{person}{R.~E. Moore}, \bibinfo{person}{R.~B.
  Kearfott}, {and} \bibinfo{person}{M.~J. Cloud}.}
  \bibinfo{year}{2009}\natexlab{}.
\newblock \bibinfo{booktitle}{\emph{Introduction to Interval Analysis}}.
\newblock \bibinfo{publisher}{SIAM}.
\newblock
\showISBNx{978-0-89871-669-6 978-0-89871-771-6}
\urldef\tempurl%
\url{https://doi.org/c8ctwd}
\showDOI{\tempurl}


\bibitem[Muller et~al\mbox{.}(2018)]%
        {MullerBrunieDeDinechinJeannerodJoldesEtAl_2018}
\bibfield{author}{\bibinfo{person}{J.-M. Muller}, \bibinfo{person}{N.~Brunie},
  \bibinfo{person}{F.~De~Dinechin}, \bibinfo{person}{C.-P. Jeannerod},
  \bibinfo{person}{M.~Joldes}, \bibinfo{person}{V.~Lef{\`e}vre},
  \bibinfo{person}{G.~Melquiond}, \bibinfo{person}{N.~Revol}, {and}
  \bibinfo{person}{S.~Torres}.} \bibinfo{year}{2018}\natexlab{}.
\newblock \bibinfo{booktitle}{\emph{Handbook of Floating-Point Arithmetic}
  (\bibinfo{edition}{2} ed.)}.
\newblock \bibinfo{publisher}{Springer}.
\newblock
\showISBNx{978-3-319-76525-9 978-3-319-76526-6}
\urldef\tempurl%
\url{https://doi.org/gtdkwj}
\showDOI{\tempurl}


\bibitem[Neumaier(2003)]%
        {Neumaier_2003}
\bibfield{author}{\bibinfo{person}{A.~Neumaier}.}
  \bibinfo{year}{2003}\natexlab{}.
\newblock \showarticletitle{Taylor Forms---{{Use}} and Limits}.
\newblock \bibinfo{journal}{\emph{Reliable Computing}} \bibinfo{volume}{9},
  \bibinfo{number}{1} (\bibinfo{year}{2003}), \bibinfo{pages}{43--79}.
\newblock
\showISSN{1573-1340}
\urldef\tempurl%
\url{https://doi.org/c48bk4}
\showDOI{\tempurl}


\bibitem[Revol and Rouillier(2023)]%
        {RevolRouillier_1999}
\bibfield{author}{\bibinfo{person}{N.~Revol} {and}
  \bibinfo{person}{F.~Rouillier}.} \bibinfo{year}{1999/2023}\natexlab{}.
\newblock \bibinfo{title}{{{MPFI}}}.
\newblock
\newblock
\urldef\tempurl%
\url{https://gitlab.inria.fr/mpfi/mpfi}
\showURL{%
\tempurl}


\bibitem[Rodriguez and Wang(2017)]%
        {RodriguezWang_2017}
\bibfield{author}{\bibinfo{person}{J.~I. Rodriguez} {and}
  \bibinfo{person}{B.~Wang}.} \bibinfo{year}{2017}\natexlab{}.
\newblock \bibinfo{title}{Numerical Computation of Braid Groups}.
\newblock
\newblock
\showeprint[arxiv]{1711.07947}


\bibitem[Rump(1983)]%
        {Rump_1983}
\bibfield{author}{\bibinfo{person}{S.~M. Rump}.}
  \bibinfo{year}{1983}\natexlab{}.
\newblock \showarticletitle{Solving Algebraic Problems with High Accuracy}.
\newblock In \bibinfo{booktitle}{\emph{A {{New Approach}} to {{Scientific
  Computation}}}}, \bibfield{editor}{\bibinfo{person}{U.~W. Kulisch} {and}
  \bibinfo{person}{W.~L. Miranker}} (Eds.). \bibinfo{publisher}{Academic
  Press}, \bibinfo{pages}{51--120}.
\newblock
\showISBNx{978-0-12-428660-3}
\urldef\tempurl%
\url{https://doi.org/10/kh8k}
\showDOI{\tempurl}


\bibitem[Sommese et~al\mbox{.}(2001)]%
        {SommeseVerscheldeWampler_2001}
\bibfield{author}{\bibinfo{person}{A.~J. Sommese},
  \bibinfo{person}{J.~Verschelde}, {and} \bibinfo{person}{C.~W. Wampler}.}
  \bibinfo{year}{2001}\natexlab{}.
\newblock \showarticletitle{Numerical Decomposition of the Solution Sets of
  Polynomial Systems into Irreducible Components}.
\newblock \bibinfo{journal}{\emph{SIAM J. Numer. Anal.}} \bibinfo{volume}{38},
  \bibinfo{number}{6} (\bibinfo{year}{2001}), \bibinfo{pages}{2022--2046}.
\newblock
\showISSN{0036-1429, 1095-7170}
\urldef\tempurl%
\url{https://doi.org/10/fv5jzk}
\showDOI{\tempurl}


\bibitem[Sommese et~al\mbox{.}(2005)]%
        {SommeseVerscheldeWampler_2005}
\bibfield{author}{\bibinfo{person}{A.~J. Sommese},
  \bibinfo{person}{J.~Verschelde}, {and} \bibinfo{person}{C.~W. Wampler}.}
  \bibinfo{year}{2005}\natexlab{}.
\newblock \showarticletitle{Introduction to Numerical Algebraic Geometry}.
\newblock In \bibinfo{booktitle}{\emph{Solving {{Polynomial Equations}}}},
  \bibfield{editor}{\bibinfo{person}{M.~Bronstein},
  \bibinfo{person}{A.~M. Cohen}, \bibinfo{person}{H.~Cohen},
  \bibinfo{person}{D.~Eisenbud}, \bibinfo{person}{B.~Sturmfels},
  \bibinfo{person}{A.~Dickenstein}, {and} \bibinfo{person}{I.~Z.
    Emiris}} (Eds.). \bibinfo{publisher}{Springer},
  \bibinfo{pages}{301--337}.
\newblock
\showISBNx{978-3-540-27357-8}
\urldef\tempurl%
\url{https://doi.org/10/bzsc24}
\showDOI{\tempurl}


\bibitem[Telen et~al\mbox{.}(2020)]%
        {TelenVanBarelVerschelde_2020}
\bibfield{author}{\bibinfo{person}{S.~Telen}, \bibinfo{person}{M.~Van~Barel},
  {and} \bibinfo{person}{J.~Verschelde}.} \bibinfo{year}{2020}\natexlab{}.
\newblock \showarticletitle{A Robust Numerical Path Tracking Algorithm for
  Polynomial Homotopy Continuation}.
\newblock \bibinfo{journal}{\emph{SIAM J. Sci. Comput.}} \bibinfo{volume}{42},
  \bibinfo{number}{6} (\bibinfo{year}{2020}), \bibinfo{pages}{A3610--A3637}.
\newblock
\showISSN{1064-8275}
\urldef\tempurl%
\url{https://doi.org/10/grqm6n}
\showDOI{\tempurl}


\bibitem[{van der Hoeven}(2015)]%
        {VanderHoeven_2015}
\bibfield{author}{\bibinfo{person}{J.~{van der Hoeven}}.}
  \bibinfo{year}{2015}\natexlab{}.
\newblock \bibinfo{title}{Reliable Homotopy Continuation}.
\newblock
\newblock
\urldef\tempurl%
\url{https://hal.science/hal-00589948v4}
\showURL{%
\tempurl}


\bibitem[Verschelde(1999)]%
        {Verschelde_1999}
\bibfield{author}{\bibinfo{person}{J.~Verschelde}.}
  \bibinfo{year}{1999}\natexlab{}.
\newblock \showarticletitle{Algorithm 795: {{PHCpack}}: {{A}} General-Purpose
  Solver for Polynomial Systems by Homotopy Continuation}.
\newblock \bibinfo{journal}{\emph{ACM Trans. Math. Softw. TOMS}}
  \bibinfo{volume}{25}, \bibinfo{number}{2} (\bibinfo{year}{1999}),
  \bibinfo{pages}{251--276}.
\newblock
\urldef\tempurl%
\url{https://doi.org/10/fncfxj}
\showDOI{\tempurl}


\bibitem[Xu et~al\mbox{.}(2018)]%
        {XuBurrYap_2018}
\bibfield{author}{\bibinfo{person}{J.~Xu}, \bibinfo{person}{M.~Burr}, {and}
  \bibinfo{person}{C.~Yap}.} \bibinfo{year}{2018}\natexlab{}.
\newblock \showarticletitle{An Approach for Certifying Homotopy Continuation
  Paths: {{Univariate}} Case}. In \bibinfo{booktitle}{\emph{Proc. {{ISSAC}}
  2018}}. \bibinfo{publisher}{ACM},
  \bibinfo{pages}{399--406}.
\newblock
\showISBNx{978-1-4503-5550-6}
\urldef\tempurl%
\url{https://doi.org/10/ggck7k}
\showDOI{\tempurl}


\end{thebibliography}


\end{document}